\documentclass[12pt]{amsart}

\usepackage{tgpagella}
\usepackage{euler}
\usepackage[T1]{fontenc}
\usepackage{amsmath, amssymb}
\usepackage[hidelinks]{hyperref}
\usepackage[english]{babel}
\usepackage{mathrsfs}
\usepackage{eucal}
\usepackage[all]{xy}
\usepackage{tikz}
\usepackage{todonotes}
\newtheorem{thm}{Theorem}[section]
\newtheorem*{thm*}{Theorem}
\newtheorem{lem}[thm]{Lemma}
\newtheorem{fact}[thm]{Fact}

\newtheorem{prop}[thm]{Proposition}
\newtheorem*{prop*}{Proposition}

\newtheorem{cor}[thm]{Corollary}
\newtheorem*{cor*}{Corollary}

\theoremstyle{definition}
\newtheorem{defn}[thm]{Definition}
\newtheorem*{defn*}{Definition}

\newtheorem{remark}[thm]{Remark}

\newtheorem{question}[thm]{Question}

\newtheorem*{question*}{Question}
\newtheorem*{Pquestion*}{Popa's question}

\newtheorem*{conv*}{Convention}

\newcommand{\R}{\mathbb{R}}
\newcommand{\C}{\mathbb{C}}
\newcommand{\N}{\mathbb{N}}

\newcommand{\ep}{\varepsilon}
\newcommand{\E}{\mathrm{E}}

\newcommand{\dminus}{ 
\buildrel\textstyle\ .\over{\hbox{ 
\vrule height3pt depth0pt width0pt}{\smash-} 
}}

\def\bb{\mathbb}

\def\bb{\mathbb}

\def\cal{\mathcal}

\def\u{\mathsf 1}

\newcommand{\cstar}{$\mathrm{C}^*$}

\def \tp{\operatorname{tp}}
\newcommand{\SA}{\operatorname{SA}}

\makeatletter

\def\dotminussym#1#2{%
  \setbox0=\hbox{$\m@th#1-$}%
  \kern.5\wd0%
  \hbox to 0pt{\hss\hbox{$\m@th#1-$}\hss}%
  \raise.6\ht0\hbox to 0pt{\hss$\m@th#1.$\hss}%
  \kern.5\wd0}
\newcommand{\dotminus}{\mathbin{\mathpalette\dotminussym{}}}

\DeclareMathOperator{\Sp}{Sp}

\newcommand\cU{{\cal U}}

\def \R{\mathcal R}
\def \u{\mathcal U}

\def\BC{\operatorname{BC}}

\allowdisplaybreaks
\makeatletter
\def\l@subsection{\@tocline{2}{0pt}{2.5pc}{5pc}{}}
\def\l@subsubsection{\@tocline{2}{0pt}{5pc}{7.5pc}{}}
\makeatother

\begin{document}

%%%%%%%%%%%%%%%%%%%%%%%%%%%%%%%%%%%%%%%%%%%%%%

\title{Model theory and Connes' bicentralizer problem}

\author{Hiroshi Ando}
\address{Department of Mathematics and Informatics, Chiba University, 1-33 Yayoi-cho, Inage, Chiba
263-8522, Japan}
\email{hiroando@math.s.chiba-u.ac.jp}
\thanks{Ando was partially supported by 
Japan Society for the Promotion of Sciences (JSPS) KAKENHI 	25K07024}
\author{Isaac Goldbring}
\address{Department of Mathematics\\University of California, Irvine, 340 Rowland Hall (Bldg.\# 400),
Irvine, CA 92697-3875}
\email{isaac@math.uci.edu}
\urladdr{http://www.math.uci.edu/~isaac}
\thanks{Goldbring was partially supported by NSF grant DMS-2054477.}

\begin{abstract}
We make a series of model-theoretic contributions to Connes' bicentralizer problem, one of the most prominent open problems in the theory of von Neumann algebras. Our work builds on the recent result of Houdayer and Marrakchi who show that, for separable diffuse W$^*$-probability spaces, having trivial bicentralizer is equivalent to being selfless, that is, having the first factor inclusion into the free product be an existential embedding. We first show that the class of selfless $W^*$-probability spaces is $\forall\exists$-axiomatizable.  We then extend the Houdayer-Marrakchi equivalence to all diffuse W$^*$-probability spaces, removing the separability hypothesis. Combining these results, we show that for any axiomatizable class of diffuse $W^*$-probability spaces, those with trivial bicentralizer form an $\forall\exists$-axiomatizable class; in particular, the class of type $\mathrm{III}_1$ factors with trivial bicentralizer is $\forall\exists$-axiomatizable. We give concrete axioms for this class using totally bounded variants of Haagerup's characterization of the bicentralizer, which we develop here and believe to be of independent interest. We also introduce the notion of pseudoperiodic $\mathrm{III}_1$ factors and show that any such factor has trivial bicentralizer.  In the final section, we prove that the bicentralizer problem has a positive solution if and only if the bicentralizer functor is a zeroset relative to the theory of $\mathrm{III}_1$ factors. We use this result to give an equivalent formulation of the bicentralizer problem in terms of a uniformity condition on Haagerup's Dixmier-type characterization of the bicentralizer.
\end{abstract}

\maketitle

\tableofcontents

\section{Introduction}

The work presented here is a model theoretic contribution to \textbf{Connes' bicentralizer problem}, one of the most prominent open problems in von Neumann algebra theory.  To explain the problem, suppose that $(M,\varphi)$ is a \textbf{W$^*$-probability space}, that is, $M$ is a $\sigma$-finite von Neumann algebra and $\varphi$ is a faithful normal state on $M$.  The \textbf{bicentralizer} of $(M,\varphi)$ is the set $\BC(M,\varphi)$ consisting of those $a\in M$ satisfying the following:  given $\epsilon>0$, there is $\delta>0$ so that, whenever $u\in M$ is a unitary for which $\|u\varphi-\varphi u\|:=\sup\{|\varphi(ub)-\varphi(bu)| \ : \ \|b\|\leq 1\}<\delta$, then one has $\|u^*au-a\|_\varphi<\epsilon$.  (This is not Connes' original definition, but rather a reformulation due to Haagerup \cite[Section 1]{haagerup87}.)  The bicentralizer $\BC(M,\varphi)$ is always a von Neumann subalgebra of $M$.  In case $M$ is a III$_1$ factor, $\BC(M,\varphi)$ is either trivial, that is, equals $\bb C$, or else is a type III$_1$ factor again.  

The relevance of the bicentralizer stems from Connes' classification of injective von Neumann algebras.  After having classified all injective factors with separable predual besides those of type III$_1$, Connes proved that an injective type III$_1$ factor $M$ with separable predual is isomorphic to the Araki-Woods factor $\mathcal R_\infty$ provided for some (or any) faithful normal state $\varphi$ on $M$, the modular flow $\sigma^{\varphi}$ is approximately inner, and he showed that this is the case if $\BC(M,\varphi)=\C$, which was confirmed later by Haagerup \cite{haagerup87}, completing the classification problem for injective factors. 
Moreover, Connes \cite{Connes85} (also mentioned in \cite[Section 3]{haagerup87}) suggested that every type III$_1$ factor with separable predual should have trivial bicentralizer; this problem is now known as Connes' bicentralizer problem. Although there are various classes of III$_1$ factors where the bicentralizers have been shown to be trivial \cite{AHHM20,Bikram2024MixedQBicentralizer,haagerup87,Houdayer09freeArakiWoods,HoudayerIsono20bicentralizerqArakiWoods,HoudayerUeda16asymptotic,Marrakchi20,Marrakchi25Kadison}, the general case is still open.

The relevance of the bicentralizer problem beyond the injective case can be best understood by Haagerup's characterization \cite[Theorem 3.1]{haagerup87} that if $M$ is a factor of type III$_1$ with separable predual, then it has trivial bicentralizer if and only if there exists a faithful normal state $\varphi$ with \textbf{large centralizer}, that is, if $M_{\varphi}'\cap M=\C$. The existence of such a state has proven to be of great importance in applications, for example, in Houdayer and Isono's works on the generalization of Popa's free independence theorem \cite{HoudayerIsono15freeindependence} and on the unique prime factorization theorem for III$_1$ factors \cite{houdayerisono17prime}. In a more recent work, the first author with Haagerup, Houdayer and Marrakchi \cite{AHHM20} constructed a flow $\beta^{\varphi}$, called the \textbf{bicentralizer flow}, on $\BC(M,\varphi)$ (or more generally, on the relative bicentralizer $\BC(N\subset M,\varphi)$) and its ergodicity is shown to be equivalent to the existence of an irreducible hyperfinite subfactor with expectation. Since then, there has been remarkable progress made by Marrakchi, where the flow is indeed shown to be ergodic \cite{Marrakchi20}, and this fact was used to show the weak relative Dixmier property for inclusion of von Neumann algebras with expectation \cite{Marrakchi21Dixmier} and applications to Kadison's problem and further generalization of the bicentralizer conjecture \cite{Marrakchi25Kadison}. Incidentally, Isono \cite{Isono24} confirmed the Haagerup--St\o rmer conjecture \cite{HaagerupStormer90,HaagerupStormer90pointwise,HaagerupStormer94pointwise} about pointwise inner automorphisms of factors of type III$_1$ under the trivial bicentralizer condition, where Marrakchi's works played a key role. All of these recent developments provide the motivation to study the bicentralizer problem beyond the injective case.    

The starting point of our work is the recent article of Houdayer and Marrakchi \cite{cyrilamine}, who proved that, for \emph{separable, diffuse} W$^*$-probability spaces, having a trivial bicentralizer is equivalent to being \textbf{selfless}. Here, a W$^*$-probability space $(M,\varphi)$ is said to be selfless if and only if there is an embedding from the free product $(M,\varphi)*(M,\varphi)$ into the (Ocneanu) ultraproduct $(M,\varphi)^\u$ such that, when restricted to the first factor, yields the usual diagonal embedding of $(M,\varphi)$ into its ultrapower. 
Selflessness was first introduced and investigated by Robert \cite{robert2025selfless}, motivated by Popa's article \cite{popa1995free}, where it was shown (using the above terminology) that any II$_1$ factor (equipped with its canonical trace) is selfless. Houdayer--Isono's free independence result mentioned above in this context is the statement that a W$^*$-probability space $(M,\varphi)$ is selfless if $M$ is a factor with $\varphi$ having large centralizer. This result played a key role in one direction of the proof of the above equivalence. Selflessness has been a very active area of study in \cstar-algebra theory over the past couple of years, stemming from the recent breakthrough result by Amrutam et. al. \cite{amrutam2025strict}, where it is shown that the reduced group \cstar-algebra C$^*_r(\bb F_2)$ has strict comparison by establishing that it is selfless (when equipped with its canonical trace), settling a longstanding open problem.

Selflessness is really a model-theoretic concept:  $(M,\varphi)$ is selfless if and only if the embedding $(M,\varphi)\subseteq (M,\varphi)*(M,\varphi)$ is existential.  (Here, we view W$^*$-probability spaces as model-theoretic structures as in \cite{AGHS25}; see also Subsection \ref{modeltheorypresentation} below.)  Existential embeddings of W$^*$-probability spaces were first studied in detail by the second author and Houdayer in \cite{goldbringhoudayer}, where it was observed that any \textbf{existentially closed (e.c.)} W$^*$-probability space (that is, any W$^*$-probability space for which any embedding into a larger W$^*$-probability space is existential) is a III$_1$ factor that tensorially absorbs $\mathcal R_\infty$, which implies that it must have trivial bicentralizer by \cite[Theorem D]{Marrakchi20}.  That e.c. W$^*$-probability spaces have trivial bicentralizer also follows from the aforementioned result of Houdayer and Marrakchi.

Our first main result is the following:
\begin{thm*}[Corollary \ref{selflessAEaxiomatizable}]
The class of selfless W$^*$-probability spaces is an $\forall\exists$-axiomatizable class.   
\end{thm*}

We do not prove this theorem by giving concrete axioms for this class, but rather prove that this class is closed under ultraproducts, ultraroots, and direct limits.  In contrast, we note that the class of e.c. W$^*$-probability spaces is not closed under ultraproducts (see \cite[Corollary 3.11]{goldbringhoudayer}).

Our next main result is to remove the separability assumption from the result of Houdayer and Marrakchi:

\begin{thm*}[Theorem \ref{generalizing}]
For any (not necessarily separable) diffuse W$^*$-probability space $(M,\varphi)$, we have that $(M,\varphi)$ is selfless if and only if $\BC(M,\varphi)=\bb C$.
\end{thm*}

The previous theorem is proven by showing that membership in either of the above classes can be detected by separable elementary subalgebras; for bicentralizers, this involves establishing that $\BC(M,\varphi)\cap N=\BC(N,\psi)$ whenever $(N,\psi)$ is an elementary (even existential) substructure of $(M,\varphi)$.  Consequently, we see that if the bicentralizer problem has a positive solution, then $\BC(M,\varphi)=\bb C$ for all III$_1$ factors $M$, regardless of the density character of $M$.

Combining the previous theorem with our first theorem yields the following:

\begin{thm*}[Corollary \ref{bicentralizeraxiomatizable}]
For any axiomatizable class $\cal K$ consisting only of diffuse W$^*$-probability spaces, the elements of $\cal K$ with trivial bicentralizer form an axiomatizable class, which is $\forall\exists$-axiomatizable if $\cal K$ is $\forall\exists$-axiomatizable.  In particular, the class of III$_1$ factors with trivial bicentralizer is axiomatizable.
\end{thm*}

In the case that $\cal K$ is the class of III$_1$ factors, we offer an alternative proof of the previous theorem that avoids the use of selflessness and instead relies on a result of Haagerup, namely that a separable III$_1$ factor $M$ with trivial bicentralizer admits a state $\varphi$ with ``large centralizer'' in the sense that $M_\varphi'\cap M=\bb C$, together with a generalization of an observation of Houdayer regarding ultrapowers of states with large centralizers. 

After having shown that the class of III$_1$ factors with trivial bicentralizer is axiomatizable, we proceed to write down concrete axioms for this class.  To do so, we introduce a variant of Haagerup's description of W$^*$-probability spaces with trivial bicentralizer (which is the description given above) that is better suited to dealing with ultraproducts and which makes use of the technology around \textbf{totally bounded} elements introduced in \cite{AGHS25}.  We believe these results should be of independent interest.

Ostensibly, the previous theorem should yield many new examples of III$_1$ factors with trivial bicentralizer.  Indeed, in analogy with the case of II$_1$ factors (see \cite[Theorem 4.3]{farah2014model}), it is conjectured that, for any separable W$^*$-probability space $(M,\varphi)$ with $M$ a III$_1$ factor, there are continuum many W$^*$-probability spaces $(M_\alpha,\varphi_\alpha)_{\alpha<2^\omega}$ elementarily equivalent to $(M,\varphi)$ but with $M_\alpha\not\cong M_\beta$ for $\alpha<\beta<2^{\omega}$ (see also \cite[Question 4.13]{goldbringhoudayer}); if $(M,\varphi)$ is one of the III$_1$ factors known to have trivial bicentralizer, then each $M_\alpha$ would also have trivial bicentralizer.

In the final section, we give a ``quantitative'' reformulation of the bicentralizer problem.  To explain this, we follow Haagerup and define, for a W$^*$-probability space $(M,\varphi)$ and $a\in M$, the quantity $$\epsilon_M(a,\delta):=\sup\{\|uau^*-a\|_\varphi \ : \ u\in \mathbf{U}(M), \|u\varphi-\varphi u\|<\delta\}.$$  Here, and throughout this paper, $\mathbf{U}(M)$ denotes the unitary group of a von Neumann algebra $M$.  It follows that $a\in \BC(M,\varphi)$ if and only if $\inf_\delta \epsilon_M(a,\delta)=0$.  We then show that the bicentralizer problem is equivalent to this description of the bicentralizer being ``uniform'' over all III$_1$ factors:  

\begin{thm*}[Theorem \ref{bicentzeroset}]
The following statements are equivalent:
\begin{enumerate}
    \item The bicentralizer problem has a positive solution.
    \item For each $\epsilon>0$, there is $\delta>0$ such that, for all W$^*$-probability spaces $(M,\varphi)$ with $M$ a III$_1$ factor, and all $a\in S_1(\BC(M,\varphi))$, we have that $\epsilon_M(a,\delta)<\epsilon$.
\end{enumerate}
\end{thm*}

The key to proving this fact is the following model-theoretic reformulation of the bicentralizer problem:

\begin{thm*}[Theorem \ref{prop zero iff trivialBC}]
The following statements are equivalent:
\begin{enumerate}
    \item The bicentralizer problem has a positive solution.
    \item The bicentralizer is a \textbf{zeroset} with respect to the theory of III$_1$ factors.  In other words:  for every family $(M_i,\varphi_i)_{i\in I}$ of W$^*$-probability spaces with each $M_i$ a III$_1$ factor and each ultrafilter $\u$ on $I$, we have that 
    $$\prod_\u \BC(M_i,\varphi_i)\subseteq \BC\left(\prod_\u (M_i,\varphi_i)\right).$$
\end{enumerate}
\end{thm*}

To prevent this paper from ballooning in size, we assume that the reader is familiar with the requisite von Neumann algebra theory and model theory; readers looking for more background on these topics that are especially relevant to the current work can consult \cite{andoultra,AH14,AGHS25,goldbringhoudayer}.  That being said, the paper contains a lengthy preliminary section, where various specific facts from both von Neumann algebra theory and model theory will be needed.  Subsection \ref{subsec Wstar prob} contains a few reminders about W$^*$-probability spaces, while Subsection \ref{subsec ultraproducts} contains the definition of the Ocneanu ultraproduct of W$^*$-probability spaces and some specific results about this ultraproduct construction we will need in the sequel.  Subsection \ref{BCsection} contains all of the information about the bicentralizer needed throughout the paper.  Subsection \ref{subsec spectralsubspace} provides some background on spectral subspaces used throughout the paper and Subsection \ref{subsec tb elements} is a technical section on totally bounded elements that is needed for our totally bounded analogues of Haagerup's results proven in Section \ref{sec tb variant of Haagerup}.  Subsection \ref{modeltheorypresentation} gives a brief description of the model-theoretic treatment of W$^*$-probability spaces while Subsection \ref{good} contains a few facts about countably incomplete and good ultrafilters used in the paper.

As mentioned above, Section \ref{sec tb variant of Haagerup} contains our totally bounded analogues of Haagerup's results while Section \ref{sec axiomatize TBC} contains the aforementioned axiomatizability results.  Section \ref{sec axiomatize TBC} also contains a subsection on what we call \textbf{pseudoperiodic III$_1$ factors}, including the fact that any pseudoperiodic III$_1$ factor has trivial bicentralizer.  Section \ref{sec BCP and zeroset} contains the aforementioned results equating the bicentralizer problem with the fact that the bicentralizer is a zeroset and with the uniform, quantitative version of Haagerup's description of the bicentralizer. 

There are also two appendices at the end of the paper.  The first contains the proof of a result from Subsection \ref{subsec spectralsubspace} relating spectral subspaces and totally bounded elements; since the proof  of this result is quite long and technical (and perhaps is even known to some experts), we postponed it to the end of the paper so as not to distract the reader from the main results.  The second appendix is a discussion about definable sets in continuous logic.  The discussion 
includes some unpublished observations of Bradd Hart, Ward Henson, and the second author regarding zerosets in continuous logic needed in connection with the results in the last section of the paper.  We include this material here and thank Hart and Henson for their permission in allowing us to do so.

This article was the result of the second author's three week visit to Chiba University as funded by a pilot Supplemental Research Collaboration Opportunity in Japan offered by The Division of Mathematical Sciences and the Office of International Science and Engineering of the National Science Foundation and the Japan Society for the Promotion of Science.  The second author would like to thank both organizations for the opportunity and to the first author and Chiba University for their hospitality during his visit.  

The authors would like to thank Jananan Arulseelan, Cyril Houdayer, and Amine Marrakchi for useful comments regarding an earlier draft of the paper.

\section{Preliminaries}

\subsection{W$^*$-probability spaces}\label{subsec Wstar prob}

In this subsection, we include a few reminders about W$^*$-probability spaces; the reader can find much more background in \cite[Section 2]{AGHS25} and the introduction to \cite{goldbringhoudayer}.

A W$^*$-\textbf{probability space} is a pair $(M, \varphi)$ that consists of a $\sigma$-finite von Neumann algebra $M$ endowed with a faithful normal state $\varphi$.  For every $x \in M$, set $$\|x\|_\varphi = \varphi(x^*x)^{1/2} \text{ and }\|x\|_\varphi^\sharp = \sqrt{\frac{\varphi(x^*x) + \varphi(xx^*)}{2}}.$$ On uniformly bounded sets, the topology induced by the norm $\|\cdot\|_\varphi$ (resp.\ $\|\cdot\|_\varphi^\#$) coincides with the strong (resp.\ $\ast$-strong) operator topology.

Given a W$^*$-probability space $(M,\varphi)$, the \textbf{centralizer} of $\varphi$, denoted $M_\varphi$, is defined by $M_\varphi:=\{x\in M \ : \ x\varphi=\varphi x\}$.  One has that $M_\varphi$ is a von Neumann subalgebra of $M$ and that $\varphi|_{M_{\varphi}}$ is a trace on $M_\varphi$.  Note that $\varphi$ is a trace on $M$ itself if and only if $M_\varphi=M$.  The state $\varphi$ is called \textbf{ergodic} if $M_\varphi=\mathbb C$.  (The terminology stems from the fact that this condition is equivalent to the modular flow $\sigma^\varphi$ being ergodic.)  

For W$^*$-probability spaces $(M, \varphi)$ and $(N, \psi)$, we say that $(M, \varphi)$ \textbf{embeds into} $(N, \psi)$, denoted $(M, \varphi) \hookrightarrow (N, \psi)$, if there exist a unital normal $\ast$-embedding $\iota : M \to N$ such that $\psi \circ \iota = \varphi$ and a faithful normal conditional expectation $E : N \to \iota(M)$ such that $\varphi \circ \iota^{-1} \circ E = \psi$.  For other characterizations of embeddings between W$^*$-probability spaces, see \cite[Lemma 2.8]{AGHS25}.

\subsection{Ultraproducts of W$^*$-probability spaces}\label{subsec ultraproducts}
Let $(M_i, \varphi_i)_{i \in I}$ be a family of $W^*$-probability spaces and let $\cU$ be an ultrafilter on $I$. 

Let $\prod_{i\in I}^{\ell^{\infty}}M_i$ denote the C$^*$-algebra of all bounded $I$-indexed sequences in $\prod_{i\in I}M_i$. 

Set 
\[\mathcal{L}_{\cU}=\left \{(x_i)_{i\in I}\in \prod_{i\in I}^{\ell^{\infty}}M_i \ : \ \lim_{i\to \cU}\|x_i\|_{\varphi_i}=0\right \} \text{ and }\mathcal{I}_{\cU}=\mathcal{L}_{\cU}\cap \mathcal{L}_{\cU}^*,\]
where $\mathcal{L}_{\cU}^*=\{x^* : x\in \mathcal{L}_{\cU}\}$. 
In other words, we have $$\mathcal{I}_{\cU}=\{(x_i)_{i\in I}\in \prod_{i\in I}^{\ell^\infty}M_i \ : \ \lim_{i\to \u}\|x_i\|_{\varphi_i}^\#=0\}.$$  Then $\mathcal{L}_{\cU}$ is a closed left ideal of $\prod_{i\in I}^{\ell^{\infty}}M_i$.  The normalizer $\mathcal{M}^{\cU}$ of $\mathcal{I}_{\cU}$ is defined as the largest $C^*$-subalgebra of $\prod_{i\in I}^{\ell^{\infty}}M_i$ in which $\mathcal{I}_{\cU}$ forms a two-sided ideal:
$$\mathcal{M}^{\cU} = \left\{ x \in {\prod_{i\in I}^{\ell^{\infty}}M_i}  \ : \  x \mathcal{I}_{\cU} \subseteq \mathcal{I}_{\cU} \text{ and } \mathcal{I}_{\cU} x \subseteq \mathcal{I}_{\cU} \right\}.$$

Finally, the \textbf{(Ocneanu) ultraproduct} of the family $(M_i, \varphi_i)_{i \in I}$ is defined as the quotient $C^*$-algebra:
$$\prod_{\cU}(M_i, \varphi_i) := \mathcal{M}^{\cU} / \mathcal{I}_{\cU}.$$

For $(x_i)_{i\in I} \in \mathcal{M}^{\cU}$, we let $(x_i)_\u$ denote its image in $\prod_\u (M_i,\varphi_i)$.  We always view $\prod_\u (M_i,\varphi_i)$ as a W$^*$-probability space by equipping it with the state $\varphi$ given by $\varphi((x_i)_\u):=\lim_\u \varphi_i(x_i)$.

If $(M,\varphi)$ is a given W$^*$-probability space and $\u$ is an ultrafilter on $I$, we write $(M,\varphi)^\u$ or $(M^\u,\varphi^\u)$ for the ultraproduct $\prod_\u (M_i,\varphi_i)$, where each $(M_i,\varphi_i)=(M,\varphi)$ and refer to this as the \textbf{(Ocneanu) ultrapower} of $(M,\varphi)$ with respect to $\u$.  There is an obvious diagonal embedding $(M,\varphi)\hookrightarrow (M^\u,\varphi^\u)$ of $(M,\varphi)$ into its ultrapower given by considering equivalence classes of constant sequences; moving forward, we always consider $M$ as a subalgebra of $M^\u$ via this embedding.  We note that the isomorphism type of $M^\u$ is independent of the choice of state $\varphi$ (all choices of state lead to the same  ideal $\mathcal{I}_\u$ and normalizer $\mathcal{M}^\u$) and so we might on occasion simply write $M^\u$ if we are only considering the underlying von Neumann algebra of the ultrapower.

The following lemma is well-known. We include the proof for completeness. 
\begin{lem}\label{lem asymptotic centralizer}  
Let $(M_i,\varphi_i)_{i\in I}$ be an $I$-indexed family of W$^*$-probability spaces and $\cU$ an  ultrafilter on $I$.  Set $(M,\varphi)=\prod_{\cU}(M_i,\varphi_i)$. If $x=(x_i)_{i\in I}\in \prod_{i\in I}^{\ell^{\infty}}M_i$ satisfies $\displaystyle \lim_{i\to \cU}\|x_i\varphi_i-\varphi_ix_i\|=0$, then $(x_i)_{i\in I}\in \mathcal{M}^{\cU}$ and $(x_i)_{\cU}\in M_{\varphi}$. 
\end{lem}

\begin{proof}
By \cite[Lemma 2.8 (b)]{haagerup87}, we have $\|x_i\xi_{\varphi_i}-\xi_{\varphi_i}x_i\|=\|x_i^*\xi_{\varphi_i}-\xi_{\varphi_i}x_i^*\|\to 0$ along $\cU$ (the result was stated for $I=\mathbb{N}$ but the same proof works in general). 
Let $y=(y_i)_{i\in I}\in \mathcal{I}_{\cU}$. To show that $xy$ and $yx$ belong to $\mathcal{I}_{\cU}$, it suffices to show that $xy\in \mathcal{L}_{\cU}^*$ and $yx\in \mathcal{L}_{\cU}$, as $\mathcal{L}_{\cU}$ (resp. $\mathcal{L}_{\cU}^*$) is a left  (resp. right) ideal of $\prod_{i\in I}^{\ell^{\infty}}M_i$. Without loss of generality, we may assume that $\|x_i\|,\,\|y_i\|\le 1$ for all $i\in I$.\\

To see that $xy\in \mathcal{L}_{\cU}^*$, we compute:
\begin{align*}
    \|y_i^*x_i^*\|_{\varphi_i}&=\|y_i^*x_i^*\xi_{\varphi_i}\|\\
    &\le \|y_i^*(x_i^*\xi_{\varphi_i}-\xi_{\varphi_i}x_i^*)\|+\|y_i^*\xi_{\varphi_i}x_i^*\|\\
    &\le \|y_i^*\|\|x_i^*\xi_{\varphi_i}-\xi_{\varphi_i}x_i^*\|+\|x_i^*\|\|y_i^*\xi_{\varphi_i}\|\\
    &\le \|x_i\xi_{\varphi_i}-\xi_{\varphi_i}x_i\|+\|y_i^*\|_{\varphi_i}\\
    &\xrightarrow{i\to \cU}0.
\end{align*}
Similarly, to see that $yx\in \mathcal{L}_{\cU}$, we compute: 
\begin{align*}
    \|y_ix_i\|_{\varphi_i}&\le \|y_i(x_i\xi_{\varphi_i}-\xi_{\varphi_i}x_i)\|+\|y_i\xi_{\varphi_i}x_i\|\\
    &\le \|x_i\xi_{\varphi_i}-\xi_{\varphi_i}x_i\|+\|y_i\xi_{\varphi_i}\|\\
    &\xrightarrow{i\to \cU}0.
\end{align*}
This shows that $x\in \mathcal{M}^{\cU}$ and thus $(x_i)_{\cU}\in M$ is defined.  We conclude that $(x_i)_{\cU}\in M_{\varphi}$ by \cite[Lemma 4.36]{AH14}. 
\end{proof}

We will also need the following fact:

\begin{prop}\label{typeofultraproduct}
Suppose that $(M_i,\varphi_i)_{i\in I}$ is a family of W$^*$-probability spaces, where $M_i$ is a type III$_{\lambda_i}$-factor.  Set $(M,\varphi):=\prod_\u (M_i,\varphi_i)$.  Suppose further that $\lambda:=\lim_\u \lambda_i>0$.  Then:  
\begin{enumerate}
    \item $M$ is a type III$_\lambda$-factor.
   \item If $\psi_i$ is another faithful normal state on $M_i$, then $\prod_\u (M_i,\varphi_i)$ is $*$-isomorphic to $\prod_\u (M_i,\psi_i)$. Moreover, if $\lambda=1$ then the isomorphism maps $(\varphi_i)_{\u}$ to $(\psi_i)_{\u}$ and thus they are isomorphic as W$^*$-probability spaces. 
\end{enumerate}  
\end{prop}

\begin{proof}
(1) Let $\widetilde{M}=\prod^{\cU}M_i$ be the Groh--Raynaud ultraproduct. Since the diameter formula for the Groh--Raynaud ultraproduct \cite[Lemma 6.10]{AH14} is valid for arbitrary ultrafilters, we get that the state space diameter $d(\widetilde{M})$ of $\widetilde{M}$ satisfies \[\displaystyle d(\widetilde{M})=\lim_{i\to \cU}d(M_i)=2\frac{1-\lambda^{\frac{1}{2}}}{1+\lambda^{\frac{1}{2}}}.\] Therefore, $\widetilde{M}$ is a type III\(_{\lambda}\) factor by \cite{CHS83diameter, HaagerupStormer90}, whence so is $M$, being a $\sigma$-finite corner of $\widetilde{M}$ by \cite[Proposition 3.15]{AH14}. The first part of Item (2) then follows from the fact that any two $\sigma$-finite projections in a type III factor are equivalent (see, for example, \cite[Proposition 1.39]{Takesaki1}). 

If $\lambda=1$, then $\widetilde{M}$ is a type III$_1$ factor. Let $p_1=\operatorname{supp}((\varphi_i)_{\u}), p_2=\operatorname{supp}((\psi_i)_{\u})\in \widetilde{M}$ be the support projections, where we regard $(\varphi_i)_{\u},(\psi_i)_{\u}\in \widetilde{M}_*$ as in \cite[Theorem 3.24]{AH14}. Choose a partial isometry $v\in \widetilde{M}$ such that $v^*v=p_1$ and $vv^*=p_2$. Then $\pi:\prod_\u (M_i,\varphi_i)\to \prod_\u (M_i,\psi_i)$ given by $\pi(x):=vxv^*$ is a $*$-isomorphism and $(\psi_i)_{\u}\circ \pi$ is a faithful normal state on $\prod_{\u}(M_i,\varphi_i)$. 
By \cite[Theorem 4.20]{AH14} (which is stated for the Groh--Raynaud ultrapower of type III$_1$ factors over a nonprincipal ultrafilter on $\N$, but the same argument, using $\lim_{i\to \u}d(M_i)=0$, works for the Groh--Raynaud ultra\emph{product} on a more general index set), there exists $w\in \mathbf{U}(\prod_{\u}(M_i,\varphi_i))$ such that $w^*(\varphi_i)_{\u}w=(\psi_i)_{\u}\circ \pi$. Then $\Theta=\pi\circ \operatorname{Ad}(w)\colon \prod_{\u}(M_i,\varphi_i)\to \prod_{\u}(M_i,\psi_i)$ is a $*$-isomorphism such that $(\psi_i)_{\u}\circ \Theta=(\varphi_i)_{\u}$. 
\end{proof}

Finally, we will often need \cite[Lemma 4.36]{AH14}, which we state here:

\begin{fact}\label{ultralimitbimodule}
Fix a family $(M_i,\varphi_i)$ of W$^*$-probability spaces and an ultrafilter $\u$ on $I$.  Set $(M,\varphi):=\prod_\u (M_i,\varphi_i)$.  Then for any $x=(x_i)_\u,y=(y_i)_\u\in M$, we have
$$\|x\varphi -\varphi y\|=\lim_\u \|x_i\varphi_i -\varphi_i y_i\|.$$
\end{fact}

\subsection{The bicentralizer}\label{BCsection}

Fix a W$^*$-probability space $(M,\varphi)$.  The \textbf{asymptotic centralizer} of $\varphi$ is the set $\operatorname{AC}(M,\varphi)$ of all bounded sequences $(x_n)$ from $M$ with the property that $\lim_{n\to \infty}\|x_n\varphi-\varphi x_n\|=0$.  Of course $M_\varphi\subset \operatorname{AC}(M,\varphi)$ (after identifying elements of $M$ with constant sequences).  The \textbf{bicentralizer of $\varphi$} is the set $\BC(M,\varphi)$ of all those $a\in M$ with the property that $\lim_n \|ax_n-x_na\|_\varphi=0$ whenever $(x_n)\in \operatorname{AC}(M,\varphi)$.  The bicentralizer $\BC(M,\varphi)$ is always a von Neumann subalgebra of $M$ with $\BC(M,\varphi)\subseteq (M_\varphi)'\cap M$ (see \cite[Proposition 1.3(1)]{haagerup87}). The following result is due to Okayasu \cite{okayasu24}: 
\begin{fact}\label{bc for generalfactors}
If $\BC(M,\varphi)=\C$, then $M$ is a factor and exactly one of the following statements holds:
\begin{itemize}
\item[(1)] $M$ is a finite factor and $\varphi$ is a tracial state. \item[(2)] $M$ is a type III$_{\lambda}$ factor for some $0<\lambda<1$ and $\varphi$ is a $\frac{2\pi}{|\log \lambda|}$-periodic state. 
\item[(3)] $M$ is a type III$_1$ factor.
\end{itemize} 
\end{fact}
Moreover, in case $M$ is a III$_1$ factor, one always has the following dichotomy:

\begin{fact}\label{bicentralizerdichotomy}
Suppose that $M$ is a $\sigma$-finite III$_1$ factor.  Then exactly one of the following possibilities holds:
\begin{enumerate}
    \item $\BC(M,\varphi)=\bb C$ for all faithful, normal states $\varphi$ on $M$.
    \item $\BC(M,\varphi)$ is a III$_1$ factor for all faithful, normal states $\varphi$ on $M$.
\end{enumerate}
\end{fact}

\textbf{Connes' bicentralizer problem} asks whether or not item (1) in the previous fact holds for every III$_1$ factor with separable predual.

  The proof of Fact \ref{bicentralizerdichotomy} uses the notion of a \textbf{self-bicentralizing state}, where a faithful normal state $\varphi$ on $M$ is called self-bicentralizing if $\BC(M,\varphi)=M$.  Suppose that $(M,\varphi)$ is a III$_1$ factor for which $\BC(M,\varphi)\not=\bb C$.  Then by \cite[Theorem 3.5]{houdayerisono17prime}, $\tilde{\varphi}:=\varphi|_{\BC(M,\varphi)}$ is a faithful normal state on a nontrivial von Neumann algebra $\tilde{M}:=\BC(M,\varphi)$ such that
$\BC(\tilde{M},\tilde{\varphi})=\tilde{M}$. A state with this property is called a self-bicentralizing state. Moreover,  $\tilde{M}_{\tilde{\varphi}}=\C$ holds, that is, $\tilde\varphi$ is an ergodic state.  By a result of Longo \cite[Proof of Theorem 3]{longo1979notes}, any nontrivial von Neumann algebra that admits a faithful, normal, ergodic state must be a III$_1$ factor.  The interested reader may consult \cite{AHHM20,Marrakchi20} for more information about self-bicentralizing states.

One of the main interests in the bicentralizer problem is the following equivalent formulation, due to Haagerup \cite[Theorem 3.1]{haagerup87}:

\begin{fact}\label{bicentralizerlargecentralizer}
For a III$_1$ factor $M$ with separable predual, the following are equivalent:
\begin{enumerate}
    \item $M$ has trivial bicentralizer.
    \item There is a faithful, normal state $\varphi$ on $M$ with \textbf{large centralizer}, that is, for which $M_\varphi'\cap M=\bb C$.
\end{enumerate}
\end{fact}

The following ultrapower characterization of the bicentralizer appears as \cite[Proposition 3.3]{houdayerisono17prime}:

\begin{fact}\label{BCultrapower}
Given a W$^*$-probability space $(M,\varphi)$ and a nonprincipal ultrafilter $\u$ on $\N$, we have that $\BC(M,\varphi)=(M^\u_{\varphi^\u})'\cap M$.
\end{fact}

We will need to use the following facts about the bicentralizer due to Haagerup.  Fix a W$^*$-probability space $(M,\varphi)$, $a\in M$, and $\delta>0$.  Define
$$\epsilon_M(a,\delta):=\sup\{\|u^*au-a\|_\varphi \ : \ u\in U(M), \|u\varphi-\varphi u\|\leq \delta\}.$$  The following fact is embedded in the proof of \cite[Lemma 1.2]{haagerup87}:

\begin{fact}\label{haagerupBCchar}
For a W$^*$-probability space $(M,\varphi)$ and $a\in M$, one has that $a\in \BC(M,\varphi)$ if and only if:  for all $\epsilon>0$, there is $\delta>0$ such that $\epsilon_M(a,\delta)<\epsilon$.
\end{fact}

The following is \cite[Proposition 1.3(2) and Remark 1.4]{haagerup87}:

\begin{fact}
For a W$^*$-probability space $(M,\varphi)$, the following are equivalent:
\begin{enumerate}
    \item $\BC(M,\varphi)=\bb C$.
    \item For every $\delta>0$, $\overline{\operatorname{conv}}\{u^*au \ : \ u\in U(M), \ \|u\varphi-\varphi u\|\leq \delta\}\cap \bb C\cdot 1\not=\emptyset$ (where the closure is taken in the $\sigma$-weak topology).
    \item For every $\delta>0$, $\varphi(a)\cdot 1\in \overline{\operatorname{conv}}\{u^*au \ : \ u\in U(M), \ \|u\varphi-\varphi u\|\leq \delta\}$.
\end{enumerate}
\end{fact}
\subsection{Arveson spectral subspaces}\label{subsec spectralsubspace}
Here, we briefly recall Arveson spectral subspaces \cite{Arveson74}. More details can be found in  \cite[Chapter XI]{takesakiII}. 
We identify the dual group $\widehat{\mathbb{R}}$ of the additive group $\mathbb{R}$ with itself. For $f\in L^1(\mathbb{R})$, we define the Fourier transform $\hat{f}$  by 
\[\hat{f}(\lambda):=\int_{\mathbb{R}}e^{it\lambda}f(t)dt,\ \ \ \ \ \lambda \in \widehat{\mathbb{R}}=\mathbb{R}.\]
We also consider the function $\sigma_f^\varphi:M\to M$ given by $\sigma_f^{\varphi}(x):=\int_{\mathbb{R}}f(t)\sigma_t^{\varphi}(x)dt$, where the integral is taken in the $\sigma$-weak sense.
\begin{itemize}
\item[(1)] 
For $x\in M$, $\text{Sp}_{\sigma^{\varphi}}(x)$ is defined by
\[\left \{\lambda \in \widehat{\mathbb{R}} \ : \ \hat{f}(\lambda)=0 \text{  for all } f\in L^1(\mathbb{R}) \text{  with  }\sigma^{\varphi}_f(x)=0\right \}.\] 
\item[(2)] The \textbf{Arveson spectrum} of $\sigma^{\varphi}$, denoted by $\text{Sp}(\sigma^{\varphi})$ is the set
\[\left \{\lambda \in \widehat{\mathbb{R}} \ : \ \hat{f}(\lambda)=0 \text{  for all } f\in L^1(\mathbb{R}) \text{  with  }\sigma^{\varphi}_f=0\right \}.\]
It is shown that ${\rm{Sp}}(\sigma^{\varphi})=\log (\sigma(\Delta_{\varphi})\setminus \{0\})$. 
\item[(3)] For a subset $E$ of $\widehat{\mathbb{R}}$, the \textbf{spectral subspace} of $\sigma^{\varphi}$ corresponding to $E$ is given by 
\[M(\sigma^{\varphi},E):=\{x\in M \ :\ \text{Sp}_{\sigma^{\varphi}}(x)\subset E\}.\]
The spectral subspaces have the following properties:
\begin{list}{}{}
\item[(i)] $M(\sigma^{\varphi},E)^*=M(\sigma^{\varphi},-E)$.
\item[(ii)] $M(\sigma^{\varphi},E)M(\sigma^{\varphi},F)\subset M(\sigma^{\varphi},\overline{E+F})$.
\item[(iii)] $\lambda \in \text{Sp}(\sigma^{\varphi})$ if and only if $M(\sigma^{\varphi},E)\neq \{0\}$ for any closed neighborhood $E$ of $\lambda$.
\item[(iv)] If $f\in L^1(\mathbb{R})$, then $\Sp_{\sigma^{\varphi}}(\sigma_f^{\varphi}(x))\subseteq \operatorname{supp}(\hat{f})\cap \Sp_{\sigma^{\varphi}}(x)$.
\end{list}
\end{itemize}
For $a>0$, the {\bf Fej\'er kernel} $F_a:\mathbb{R}\to \mathbb{R}$ is defined by 
\[F_a(t):=\begin{cases}\dfrac{1-\cos (at)}{\pi at^2} & t\neq 0\\
\ \ \ a/2\pi & t=0. \end{cases}\]
Its Fourier transform is given by
\[\widehat{F_a}(\lambda)=\int_{\mathbb{R}}e^{it\lambda}F_a(t)dt=\begin{cases}1-\dfrac{|\lambda|}{a} & |\lambda|\le a\\
 \ \ \ 0 & |\lambda|>a.\end{cases}\]
It holds that $0\le F_a$ and $\|F_a\|_1=\widehat{F}_a(0)=1$.
In particular, we have $\sigma_{F_a}^{\varphi}(x)\in M(\sigma^{\varphi},[-a,a])$ for every $a>0$ and $x\in M$. This fact will be repeatedly used in the sequel. 

The {\bf de la Vall\'ee Poussin Kernel} $D_a:\mathbb{R}\to \mathbb{R}$ is given by
\[
D_a(t)=2F_{2a}(t)-F_a(t)=\begin{cases}
\dfrac{\cos (at)-\cos (2at)}{\pi at^2} & t\neq 0\\
\ \ \ \ \ 3a/2\pi & t=0.
\end{cases}
\]
Its Fourier transform is given by
\[\widehat{D_a}(\lambda)=\begin{cases}
1 & |\lambda|\le a\\
2-\frac{|\lambda|}{a}  & a\le |\lambda|\le 2a\\
0 & |\lambda|>2a.
\end{cases}\]
More generally, for $0<b<a$, the function $D_{a,b}=\frac{1}{a-b}(aF_a-bF_b)$ belongs to $C_b(\mathbb R)\cap L^1(\mathbb R)$ and satisfies $\widehat{D_{a,b}}=1$ on $[-b,b]$ and $\operatorname{supp}(\widehat{D_{a,b}})\subset [-a,a]$.
\subsection{On totally bounded elements}\label{subsec tb elements}

For our purposes, we will need variants of the aforementioned results of Haagerup that interact better with ultraproducts.  The key technical notion involved in these variants is that of a totally bounded element as introduced in \cite{AGHS25}:
\begin{defn}
 Suppose that $(M,\varphi)$ is a $W^*$-probability space. 
    For $K>0$, an element $a\in M$ is called \textbf{right $K$-bounded} if $\|ba\|_{\varphi}\le K\|b\|_{\varphi}$ for all $b\in M$.  If both $a$ and $a^*$ are right $K$-bounded and $\|a\|\le K$, then $a$ is called \textbf{totally $K$-bounded}. An element which is totally $K$-bounded for some $K>0$ is called \textbf{totally bounded}.    
\end{defn}
The next lemma is essentially due to Connes \cite{Connes73}. 
\begin{lem}\label{lem Kbounded}
    Let $(M,\varphi)$ be a $W^*$-probability space. Fix $K>0$ and $a\in M$. 
    Then the following are equivalent:
    \begin{itemize}
        \item[{\rm (i)}] $a$ is right $K$-bounded. 
        \item[{\rm (ii)}] The map $\mathbb{R}\ni t\mapsto \sigma_t^{\varphi}(a^*)\in M$ extends to an $M$-valued bounded continuous function on $\overline{D}_{\frac12}=\{z\in \C\mid -\frac12 \le \operatorname{Im} z\le 0\}$ which is holomorphic in the interior and such that $\|\sigma_{-i/2}^{\varphi}(a^*)\|\le K^2$.
        \item[{\rm (iii)}] $\varphi(a^*xa)\le K^2\varphi(x)$ for every $x\in M_+$.  
    \end{itemize} 
\end{lem}
\begin{proof}It is clear that (i)$\iff$(iii) holds. 
    The equivalence (ii)$\iff$(iii) follows from \cite[Lemma 4]{Connes73} (cf. \cite[Lemma VIII.3.18]{takesakiII}).  
\end{proof}

Throughout this paper, for any von Neumann algebra $M$, we let $\mathbf{U}(M)$ denote its unitary group.

\begin{cor}\label{totallyoneboundedunitary}
    Let $(M,\varphi)$ be a W$^*$-probability space and let $u\in \mathbf{U}(M)$. Then $u$ is totally 1-bounded if and only if $u\in \mathbf{U}(M_{\varphi})$.
\end{cor}
\begin{proof}
    If $u\in \mathbf{U}(M_{\varphi})$, then $u\varphi u^*=u^*\varphi u=\varphi$ and thus $u$ and $u^*$ are right 1-bounded.  
    Conversely, if $u$ is totally 1-bounded, then $u\varphi u^*\le \varphi$ and $u^*\varphi u\le \varphi$ holds, whence $u\varphi u^*=\varphi$ holds. This shows that $u\in \mathbf{U}(M_{\varphi})$. 
\end{proof}

We will also need the following facts about totally bounded elements in ultraproducts proven in \cite{AGHS25}, the second of which is fairly nontrivial:

\begin{fact}\label{TBelementsinultraproducts}
Fix a family $(M_i,\varphi_i)_{i\in I}$ of W$^*$-probability spaces and $K\geq 1$.
\begin{enumerate}
    \item If $a_i\in M_i$ is totally $K$-bounded for each $i\in I$, then the sequence $(a_i)_{i\in I}$ represents an element $a=(a_i)_\u$ in $\prod_\u (M_i,\varphi_i)$ which is itself totally $K$-bounded.
    \item Conversely, if $a\in \prod_\u (M_i,\varphi_i)$ is totally $K$-bounded, then there are totally $K$-bounded $a_i\in M_i$ such that $a=(a_i)_\u$.
\end{enumerate}
\end{fact}

The following spectral subspace criterion for proving total boundedness of an element will prove useful in the sequel.

\begin{prop}\label{prop compactspec}
Let $(M,\varphi)$ be a W$^*$-probability space. Fix $x\in M$ with $\|x\|\le 1$ and $a>0$. Consider the following three conditions: 
\begin{itemize}
    \item[{\rm (i)}] $x\in M(\sigma^{\varphi},[-a,a])$.
    \item[{\rm (ii)}] The map $\mathbb{R}\ni t\mapsto \sigma_t^{\varphi}(x)\in M$ extends to an $M$-valued entire analytic function such that 
    $\|\sigma_z^{\varphi}(x)\|\le e^{a|\operatorname{Im} z|}\|x\|$ for all $z\in \C$.
    \item[{\rm (iii)}] $x$ is totally $e^{a/2}$-bounded.  
\end{itemize}
Then (i)$\iff$(ii)$\implies$(iii). 
\end{prop}
The equivalence (i)$\iff$(ii) can be seen as an $M$-valued Paley--Wiener--Schwartz type result, and as such it might be known to experts: indeed, (ii)$\implies$(i) is due to Haagerup \cite[Lemma 2.5]{haagerup87}, and from the inequality in the proof of \cite[Lemma 4.2]{Haagerup79}, it is straightforward to see that elements of the form $x=\sigma_{F_a}^{\varphi}(y)$ for $a>0$ satisfy (ii), where $F_a$ is the Fej\'er kernel defined in $\S$\ref{subsec spectralsubspace} (note that $x\in M(\sigma^{\varphi},[-a,a])$ in this case). Also, it does not seem to us that this type of smoothing argument by summability kernels alone gives a short proof of (i)$\implies$(ii) in full generality. An argument using distributions seems inevitable in a way similar to Haagerup's approach to \cite[Lemma 2.5]{haagerup87}. The closest result we found is Matsumoto's Paley--Wiener--Schwartz type theorem in the C$^*$-algebraic setting \cite[Theorem 9.4]{Matsumoto1997}. One can indeed use the fact that $x\in M(\sigma^{\varphi},[-a,a])$ implies $t\mapsto \sigma_{t}^{\varphi}(x)$ is norm continuous, and then pass to the C$^*$-subalgebra $A=C^*(\{\sigma_t^{\varphi}(x)\mid t\in \mathbb R\})$ on which $\sigma^{\varphi}$ defines a point-norm continuous flow.  Matsumoto's above mentioned theorem then gives the following estimate: there exists $N\in \N$ and $\gamma>0$ such that 
\[\|\sigma_z^{\varphi}(x)\|\le \gamma(1+|z|)^Ne^{a|\,\mathrm{Im}z\,|},\,\,z\in \C.\]
However, we still need to remove the polynomial factor; this can be done using a Phragm\'en--Lindel\"of type estimate in the half-plane. Since the precise norm estimate is needed in our later analysis, we give a self-contained proof of this equivalence (especially, (i)$\implies$(ii)) in Appendix \ref{appendixA}.

We also would like to point out that it is quite straightforward to see that the direction (iii)$\implies$(i) cannot be reversed in general.  In fact, in a forthcoming paper, we will prove that, unlike the set of totally bounded elements, spectral subspaces do not commute with ultrapowers when \(M\) is a type III\(_1\) factor.

Although we will not explicitly need the next result, it is an immediate consequence of Proposition \ref{prop compactspec} and seems worth recording:

\begin{cor}\label{unitaryperturbtotallybounded}
    Let $(M,\varphi)$ be a W$^*$-probability space. Then for each $u\in \mathbf{U}(M)$ and $\varepsilon>0$, there exists $v\in \mathbf{U}(M)$ which is totally bounded such that $\|u-v\|_{\varphi}<\ep$.  
\end{cor}
\begin{proof}
    Take $h\in M_{{\rm sa}}$ such that $u=e^{ih}$. Since the totally bounded elements are $*$-strongly dense in $M$, there exists $h_0\in M_{\rm sa}$ which is totally bounded such that $\|u-e^{ih_0}\|_{\varphi}<\ep$ holds.  Set $v:=e^{ih_0}$. 
    Since $\sigma_t^{\varphi}(v)=e^{i\sigma_t^{\varphi}(h_0)}$ extends to an $M$-valued bounded continuous function on $\overline{D}_{\frac12}$ which is holomorphic in the interior such that $\|\sigma_{-i/2}^{\varphi}(v)\|\le e^{\|\sigma_{-i/2}^{\varphi}(h_0)\|}$, we have that $v$ is right bounded. Similarly, $v^*$ is right bounded and thus $v$ is totally bounded.  
\end{proof}

\subsection{W$^*$-probability spaces as model-theoretic structures}\label{modeltheorypresentation}

In this subsection, we briefly recall the treatment of W$^*$-probability spaces as model-theoretic structures due to Arulseelan, Hart, Sinclair, and the second author as presented in \cite{AGHS25}.  (We remark that this class was first treated model-theoretically by Dabrowski in \cite{dabrowski2019continuous}, but we prefer the treatment in \cite{AGHS25} as it allows one to speak of actual multiplication rather than the ``smeared'' multiplication used in \cite{dabrowski2019continuous}.)

The key insight into the model-theoretic treatment of W$^*$-probability spaces as presented in \cite{AGHS25} is the choice of sorts $S_K$ of totally $K$-bounded elements, as $K$ varies over $\bb N$ (as opposed to sorts being the operator norm balls as in Dabrowski's approach \cite{dabrowski2019continuous}), where each sort is equipped with the metric corresponding to the norm $\|\cdot\|_\varphi^\#$.  (Note that the sorts are indeed complete with respect to this metric.)  Since multiplication (as a two-variable function) is uniformly continuous when restricted to each $S_K$, one may allow multiplication on each sort as a distinguished binary function symbol.  Moreover, since the union of the sorts is dense in the W$^*$-probability space (as mentioned earlier), this allows one to recover the W$^*$-probability space from its ``dissected'' version.  In fact, the map which sends a W$^*$-probability space to its dissection is an equivalence of categories for which the Ocneanu ultraproduct corresponds to the model-theoretic ultraproduct.  We stress that the proof of this theorem is nontrivial and uses ideas from Tomita-Takesaki theory (and the bounded operator approach to this theory developed by Rieffel and van Daele) as well as ideas of Kadison.  For more details on this, the reader may consult \cite{AGHS25}.

\subsection{Countably incomplete and good ultrafilters}\label{good}

In a few places in the paper, we need the notion of a \textbf{good ultrafilter}.  The definition is quite technical, so we do not give it here.  We simply explain the properties needed for the applications in this paper.  We refer the reader to \cite[Chapter 8]{goldbring2022ultrafilters} for complete details in the setting of classical logic and to \cite{keisler2024using} for a treatment in continuous logic.

Recall that an ultrafilter $\u$ is called \textbf{countably incomplete} if there is a  descending family $(A_n)_{n\in \bb N}$ of elements of $\u$ with $\bigcap_{n\in \bb N}A_n=\emptyset$.  The key point is that whenever $(\cal M_i)_{i\in I}$ is a family of structures in a countable (or separable, in the continuous setting) language and $\u$ is a countably incomplete ultrafilter on $I$, then the ultraproduct $\prod_\u \cal M_i$ is $\aleph_1$-saturated.

In order to get higher levels of saturation, one needs to work with special kinds of countably incomplete ultrafilters called good ultrafilters.  More precisely, given any cardinal $\kappa$, one can define the notion of a \textbf{$\kappa^+$-good ultrafilter} $\u$ on a set $I$, which has the key property that whenever $(\cal M_i)_{i\in I}$ is a family of structures in a language of cardinality (or density character, in the continuous setting) at most $\kappa$, then the ultraproduct $\prod_\u \cal M_i$ is $\kappa^+$-saturated.  A particular consequence of this saturation is that any structure of cardinality (or density character) at most $\kappa^+$ elementarily equivalent to $\prod_\u \cal M_i$ embeds elementarily into $\prod_\u \cal M_i$.  We stress that, in ZFC, one can prove that $\kappa^+$-good ultrafilters exist for every cardinal $\kappa$.

\section{Totally bounded variants of Haagerup's results on bicentralizers}\label{sec tb variant of Haagerup}

In this section, we prove some totally bounded variants of Haagerup's results around bicentralizers introduced in Subsection 2.3 that will be useful in the sequel.

For each $\delta>0$, let $\tilde{\mathbf{U}}_{1+\delta}(M)$ denote the set of all $v\in M$ which are totally $(1+\delta)$-bounded and such that $\max\{\|v^*v-1\|_{\varphi},\|vv^*-1\|_{\varphi}\}\le \delta$ holds.

Recall the definition of Haagerup's function $\epsilon_M(a,\delta)$ from Section \ref{BCsection}: $$\ep_{M}(a,\delta)=\sup \left \{\|u^*au-a\|_{\varphi}\colon u\in \mathbf{U}(M),\,\|u\varphi-\varphi u\|\leq \delta\right \}.$$  Replacing the condition ``$u$ is a unitary for which $\|u\varphi-\varphi u\|\leq \delta$'' with ``$u\in \tilde{\mathbf{U}}_{1+\delta}(M)$'', we obtain the following function:
\begin{defn}
Let $(M,\varphi)$ be a W$^*$-probability space.  For $a\in M$ and $\delta>0$, set    
\[\tilde{\ep}_{M}(a,\delta)=\sup \left \{\|u^*au-a\|_{\varphi}\colon u\in \tilde{\mathbf{U}}_{1+\delta}(M)\right \}.\]
\end{defn}
Let $(M,\varphi)$ be a W$^*$-probability space. Let $a\in M$ and $\delta>0$. Consider the set 
$$\widetilde{C}_{\varphi}(a,\delta):=\overline{\operatorname{co}}\{u^*au\mid u\in \tilde{\mathbf{U}}_{1+\delta}(M)\},$$
where the bar denotes the $\sigma$-weak closure. 
We will show the following totally bounded version of Haagerup's characterization of trivial bicentralizer \cite[Proposition 1.3]{haagerup87}: 
\begin{prop}\label{prop tbprop1.3}
    Let $(M,\varphi)$ be a W$^*$-probability space. The following three conditions are equivalent: 
    \begin{enumerate}
    \item[(a)] $\BC(M,\varphi)=\mathbb{C}.$
    \item[(b)] For every $a\in M$ and $\delta>0$, the condition 
    $$\widetilde{C}_{\varphi}(a,\delta)\cap \C \neq \emptyset$$
    holds. 
    \item[(c)] For every $a\in M$ and $\delta>0$, the condition 
    $$\varphi (a)1\in \widetilde{C}_{\varphi}(a,\delta)$$
    holds.
  \end{enumerate}
\end{prop}
We need some preparation. We will need the following analogue of Haagerup's characterization of the bicentralizer (Fact \ref{haagerupBCchar} above):

\begin{lem}\label{lem tildeepsilon and BC}
    Let $(M,\varphi)$ be a W$^*$-probability space. Then for every $a\in M$, we have
    \[a\in \BC(M,\varphi)\iff \lim_{\delta\to 0}\tilde{\ep}_{M}(a,\delta)=0.\]
\end{lem}

To prove the previous lemma, we will need the following lemma, which will also prove important in later sections:

\begin{lem}\label{lem representing u by 1+deltabdd}
    Let $(M,\varphi)$ be a W$^*$-probability space and $\u$ an ultrafilter on an index set $I$.  Fix $u\in \mathbf{U}(M^{\cU}_{\varphi^{\cU}})$. Then for each $\delta>0$, there exists a sequence of contractions $(u_i)_{i\in I}$ in $\tilde{\mathbf{U}}_{1+\delta}(M)$ such that $u=(u_i)_{\cU}$.
\end{lem}
\begin{proof}
    Let $(\tilde{u}_i)_{i\in I}$ be a sequence of unitaries in $M$ so that $u=(\tilde{u}_i)_\u$. 
    Choose $a>0$ such that $e^{a/2}<1+\delta$ and consider the Fej\'er kernel $F_a$ defined in Subsection 2.4. Since $\int_{\mathbb R}F_a(t)dt=1$, we have, using \cite[Lemma 4.14]{AH14}, that
    \begin{align*}
    u&=\int_{\mathbb R}F_a(t)\sigma_t^{\varphi^{\cU}}(u)dt\\
    &=\sigma_{F_a}^{\varphi^{\cU}}(u)\\
    &=(\sigma_{F_a}^{\varphi}(\tilde{u}_i))_{\cU}.
    \end{align*}
    (We note that \cite[Lemma 4.14]{AH14} was stated for only countable index sets; however, the proof shows that the result holds for arbitrary index sets.)  For each $i\in I$, set $u_i=\sigma_{F_a}^{\varphi}(\tilde{u}_i)$. Then $\|u_i\|\le \|\tilde{u}_i\|=1$.  Moreover, since $\operatorname{supp}(\hat F_a)\subseteq [-a,a]$, we have that $u_i\in M(\sigma^{\varphi},[-a,a])$.  By Proposition \ref{prop compactspec}, for all $i$, $u_i$ is totally $e^{a/2}$-bounded and thus totally $(1+\delta)$-bounded. Moreover, since 
    \[0=\max(\|u^*u-1\|_{\varphi^{\cU}},\|uu^*-1\|_{\varphi^{\cU}})=\lim_{\cU}\max(\|u_i^*u_i-1\|_{\varphi},\|u_iu_i^*-1\|_{\varphi}),\]
    we have that $u_i\in \tilde{\mathbf{U}}_{1+\delta}(M)$ for $\cU$-almost every $i$.  By setting $u_i=1$ for the remaining $i$ (a $\cU$-null set), we have that $u_i\in \tilde{\mathbf{U}}_{1+\delta}(M)$ for every $i$. 
\end{proof}

\begin{remark}
    In some sense, the previous lemma is optimal.  Indeed, let $\varphi$ be an ergodic state on a III$_1$ factor $M$ with separable predual (such a state exists by the main result of \cite{marrakchivaes2024ergodic}). 
Take $u\in \mathbf{U}(M^{\cU}_{\varphi^{\cU}})\setminus \C$; such an element exists as $M^{\cU}_{\varphi^{\cU}}$ is a II$_1$ factor.  By Corollary \ref{totallyoneboundedunitary}, $u$ is totally 1-bounded in $M^{\cU}$. Since $S_1(M^{\cU})=S_1(M)^{\cU}$, $u$ has a representing sequence $(v_n)_n$ of totally 1-bounded elements in $M$, that is, $u=(v_n)_{\cU}$. However, since $\mathbf{U}(M_{\varphi})=1$ and $u\notin \C$, such ``almost unitaries'' $v_n$ cannot be taken to be unitaries in $M_\varphi$.
\end{remark}

\begin{proof}[Proof of Lemma \ref{lem tildeepsilon and BC}]
    For the forward direction, assume that there exists $\ep>0$ such that $\tilde{\ep}_M(a,\tfrac{1}{n})\ge \ep$ for every $n\in \N$.  For each $n\geq 1$, choose $u_n\in \tilde{\mathbf{U}}_{1+1/n}(M)$ such that $\|u_n^*au_n-a\|_{\varphi}\ge \ep-1/n$. 
    Then for each $n$ and $x_n\in M_+$, we have 
    \[\varphi(u_nx_nu_n^*)\le (1+\tfrac{1}{n})^2\varphi(x_n) \text{ and }
    \varphi(u_n^*x_nu_n)\le (1+\tfrac{1}{n
    })^2\varphi(x_n).\]
    Fix a nonprincipal ultrafilter $\u$ on $\N$. By Fact \ref{TBelementsinultraproducts}, we are entitled to consider $u=(u_n)_{\cU}\in M^{\cU}$. Then $u$ is a totally 1-bounded unitary in $M^{\cU}$, whence $u\in M^{\mathcal{U}}_{\varphi^{\mathcal{U}}}$ by Corollary \ref{totallyoneboundedunitary}.  Moreover, $\|u^*au-a\|_{\varphi^{\mathcal{U}}}\ge \ep$ holds. By Fact \ref{BCultrapower}, this shows that $a\notin \BC(M,\varphi)=(M^{\mathcal{U}}_{\varphi^{\mathcal{U}}})'\cap M$.\\
    
    To prove the converse, assume, towards a contradiction, that $\lim_{\delta\to 0}\tilde{\ep
    }_M(a,\delta)=0$ and yet $a\notin \BC(M,\varphi)$. By Fact \ref{haagerupBCchar}, $\displaystyle \ep:=\lim_{\delta\to 0}\ep_M(a,\delta)>0$. Choose $\delta>0$ such that $\tilde{\ep}_M(a,\delta)<\frac{\ep}{2}$. By the definition of $\epsilon$, for every $n\in \N$, there exists $u_n\in \mathbf{U}(M)$ such that $\|u_n\varphi-\varphi u_n\|<\frac{1}{n}$ and $\|u_n^*au_n-a\|_{\varphi}\ge \ep$. Then $u=(u_n)_{\mathcal{U}}\in \mathbf{U}(M^{\mathcal{U}}_{\varphi^{\mathcal{U}}})$ by Lemma \ref{lem asymptotic centralizer}. 
    By Lemma \ref{lem representing u by 1+deltabdd}, there is a sequence $(y_n)_{n\in \N}$ of contractions in $\tilde{\mathbf{U}}_{1+\delta}(M)$ such that $u=(y_n)_\u$.  Moreover, since $$\lim_{ \cU}\|y_n^*ay_n-a\|_{\varphi}=\|u^*au-a\|_{\varphi^{\cU}}\geq \varepsilon,$$ we have $\|y_n^*ay_n-a\|_{\varphi}\ge \frac{\ep}{2}$ for $\cU$-almost all $n$, which contradicts $\tilde{\ep}_M(a,\delta)<\ep/2$.
\end{proof}

\begin{lem}\label{lem tildeU is multiplicative}
Suppose $(M,\varphi)$ is a W$^*$-probability space and $\delta_1,\delta_2>0$. Then $$\tilde{\mathbf{U}}_{1+\delta_1}(M)\cdot\tilde{\mathbf{U}}_{1+\delta_2}(M)\subset \tilde{\mathbf{U}}_{1+\delta}(M),$$ where $\delta:=\delta_1+\delta_2+\delta_1\delta_2$. 
\end{lem}
\begin{proof}
    Take $u\in \tilde{\mathbf{U}}_{1+\delta_1}(M)$ and $v\in \tilde{\mathbf{U}}_{1+\delta_2}(M)$, and set $w:=uv$. Then $$w\in S_{1+\delta_1}(M)S_{1+\delta_2}(M)\subseteq S_{(1+\delta_1)(1+\delta_2)}(M)=S_{1+\delta}(M).$$ Moreover, 
    \begin{align*}
        \|w^*w-1\|_{\varphi}&=\|v^*u^*uv-1\|_{\varphi}\leq\|v^*(u^*u-1)v\|_{\varphi}+\|v^*v-1\|_{\varphi},\\
        \|v^*(u^*u-1)v\|_{\varphi}^2&=\varphi(v^*(u^*u-1)^*(u^*u-1)v)\\
        &\le (1+\delta_2)^2\varphi((u^*u-1)^*(u^*u-1))\le (1+\delta_2)^2\delta_1^2,
    \end{align*}
    and therefore 
    $$\|w^*w-1\|_{\varphi}\le (1+\delta_2)\delta_1+\delta_2=\delta.$$
    Similarly, $\|ww^*-1\|_{\varphi}\le \delta$ holds. This shows that $w\in \tilde{\mathbf{U}}_{1+\delta}(M)$. 
\end{proof}
The next lemma is a totally bounded version of \cite[Lemma 1.2]{haagerup87}. 
\begin{lem}\label{lem Haageruptb1.2}
Suppose $(M,\varphi)$ is a W$^*$-probability space and $a\in M$. Then 
$$a\in \BC(M,\varphi)\iff \bigcap_{\delta>0}\widetilde{C}_{\varphi}(a,\delta)=\{a\}.$$
\end{lem}
\begin{proof}
    We first note that if $u\in \tilde{\mathbf{U}}_{1+\delta}(M)$, then $\|u^*au-a\|_{\varphi}\le \tilde{\ep}_{M}(a,\delta)$. This implies that 
    \begin{equation}\|x-a\|_{\varphi}\le \tilde{\ep}_{M}(a,\delta)\label{eq x-aphi}
    \end{equation}
    for every $x$ in the convex hull $\widetilde{C}_{\varphi}^0(a,\delta)=\operatorname{co}\{u^*au \ : \  u\in \tilde{\mathbf{U}}_{1+\delta}(M)\}$, and thus for every $x\in \widetilde{C}_{\varphi}(a,\delta)$ as $\widetilde{C}_{\varphi}^0(a,\delta)$ is a bounded, convex, $\sigma$-weakly (hence $\sigma$-*strongly) dense subset of $\widetilde{C}_{\varphi}(a,\delta)$.\\
    
    First assume that $a\in \BC(M,\varphi)$. Then by Lemma \ref{lem tildeepsilon and BC}, $\displaystyle \lim_{\delta\to 0}\tilde{\ep}_M(a,\delta)=0$ and thus $\bigcap_{\delta>0}\widetilde{C}_{\varphi}(a,\delta)=\{a\}$  holds by (\ref{eq x-aphi}).\\
    
    Now assume that $a\notin \BC(M,\varphi)$ and thus $\ep:=\lim_{\delta\to 0}\tilde{\ep}_{M}(a,\delta)>0$. Then for each $n\in \mathbb{N}$, there exists $u_n\in \tilde{\mathbf{U}}_{1+\frac{1}{n}}(M)$ such that $\|u_n^*au_n-a\|_{\varphi}\ge \frac{\ep}{2}$. Fix a nonprincipal ultrafilter $\cU$ on $\mathbb{N}$ and set $b:=\lim_{ \cU}u_n^*au_n$ ($\sigma$-weak limit). Given $\delta>0$, choose $n_0\in \mathbb{N}$ such that $1/n_0<\delta$. For any $\sigma$-weak open neighborhood $W$ of $b$ in $M$, there exists $\cU$-many $n>n_0$ such that $u_n^*au_n\in W\cap \widetilde{C}_{\varphi}(a,\frac{1}{n})\subseteq W\cap \widetilde{C}_{\varphi}(a,\delta)$. Therefore, $b\in \widetilde{C}_{\varphi}(a,\delta)$ holds.  This shows that $b\in \bigcap_{\delta>0}\widetilde{C}_{\varphi}(a,\delta)$. We show that $b\neq a$. Note that the condition $u_n\in \tilde{\mathbf{U}}_{1+\frac{1}{n}}(M)$ implies that $(u_n)_{n\in \mathbb{N}}$ defines an element $u=(u_n)_{\cU}$, which is a totally 1-bounded unitary, and thus $u\in \mathbf{U}(M^{\cU}_{\varphi^{\cU}})$ holds by Corollary \ref{totallyoneboundedunitary}. In particular, we have 
    \begin{align*}
        \lim_{n\to \cU}\|u_n^*au_n\|_{\varphi}^2&=\lim_{n\to \cU}\varphi(u_n^*a^*au_n)=\varphi^{\cU}(u^*a^*au)\\
        &=\varphi^{\cU}(a^*a)=\|a\|_{\varphi}^2.
    \end{align*}
    On the other hand, the equation 
    \begin{align*}
        2\operatorname{Re}\varphi(a^*u_n^*au_n)&=\|a\|_{\varphi}^2+\|u_n^*au_n\|_{\varphi}^2-\|a-u_n^*au_n\|_{\varphi}^2
    \end{align*}
    implies that 
    \begin{align*}
        2\operatorname{Re}\varphi(a^*b)&=\lim_{n\to \cU}2\operatorname{Re}\varphi(a^*u_n^*au_n)\\
        &=2\|a\|_{\varphi}^2-\lim_{n\to \cU}\|a-u_n^*au_n\|_{\varphi}^2\\
        &\le 2\|a\|_{\varphi}^2-\frac{\ep^2}{4}.
    \end{align*}
    Therefore, $a\neq b$, as desired. 
\end{proof}
\begin{proof}[Proof of Proposition \ref{prop tbprop1.3}]
(a)$\implies$(b):  Take $a\in M$ and set $\widetilde{C}_{\varphi}(a):=\bigcap_{\delta>0}\widetilde{C}_{\varphi}(a,\delta)$, which is a $\sigma$-weakly compact convex subset of $M$, whence it is nonempty since $a\in \widetilde{C}_{\varphi}(a,\delta)$ for each $\delta>0$. Let $\Lambda_{\varphi}\colon M\to L^2(M,\varphi)$ be the canonical embedding, namely $\Lambda_\varphi(x):=x\xi_\varphi$, which is a WOT-weak continuous linear map. Therefore, $\Lambda_{\varphi}(\widetilde{C}_{\varphi}(a))$ is a nonempty weakly (hence norm) closed convex subset of $L^2(M,\varphi)$. Therefore, there exists $b\in \widetilde{C}_{\varphi}(a)$ such that $\|x\|_{\varphi}>\|b\|_{\varphi}$ for every $x\in \widetilde{C}_{\varphi}(a)\setminus \{b\}$.  We show $b\in \BC(M,\varphi)$. Fix $\delta>0$. Let $\delta_0$ be such that $(1+\delta_0)^2=1+\delta$.  If $u,v\in \tilde{\mathbf{U}}_{1+\delta_0}(M)$, then $vu\in \tilde{\mathbf{U}}_{1+\delta}(M)$ by Lemma \ref{lem tildeU is multiplicative}. Therefore, $u^*v^*avu\in \widetilde{C}_{\varphi}(a,\delta)$. Since $b\in \widetilde{C}_{\varphi}(a,\delta_0)$ and  $\widetilde{C}_{\varphi}(a,\delta)$ is $\sigma$-strongly closed and convex, we have 
$$u^*bu\in \widetilde{C}_{\varphi}(a,\delta) \text{ for }u\in \tilde{\mathbf{U}}_{1+\delta_0}(M).$$
Taking the $\sigma$-strong closed convex hull of those elements in turn shows 
$$\widetilde{C}_{\varphi}(b)\subseteq \widetilde{C}_{\varphi}(b,\delta_0)\subseteq \widetilde{C}_{\varphi}(a,\delta).$$
Since $\delta>0$ is arbitrary, we obtain $\widetilde{C}_{\varphi}(b)\subseteq \widetilde{C}_{\varphi}(a)$. On the other hand, for each $\delta>0$ and $u\in \tilde{\mathbf{U}}_{1+\delta}(M)$,  we have 
\begin{align*}
    \|u^*bu\|_{\varphi}^2&=\varphi(u^*b^*uu^*bu)\le (1+\delta)^2\varphi(b^*uu^*b)\\
    &\le (1+\delta)^4\varphi(b^*b).
\end{align*}
Therefore, $\|u^*bu\|_{\varphi}\le (1+\delta)^2\|b\|_{\varphi}$. Since $\widetilde{C}_{\varphi}(b,\delta)$ is $\sigma$-strongly closed and convex, it follows that $\|x\|_{\varphi}\le (1+\delta)^2\|b\|_{\varphi}$ for every $x\in \widetilde{C}_{\varphi}(b,\delta)$, and thus $\|x\|_{\varphi}\le \|b\|_{\varphi}$ for every $x\in \widetilde{C}_{\varphi}(b)$, whence $x=b$. That is, $\widetilde{C}_{\varphi}(b)=\{b\}$, which by Lemma \ref{lem Haageruptb1.2}, shows that $b\in \BC(M,\varphi)$. By (a), we have $b\in \C$, so $b\in \widetilde{C}_{\varphi}(a,\delta)\cap \C$ for every $\delta>0$. This shows (b).\\

To show (b)$\implies$(c), we first show that any $x\in \widetilde{C}_{\varphi}(a)$ satisfies $\varphi(x)=\varphi(a)$.  To this end,
we show that for every $\eta>0$, there exists $\delta>0$ such that, for all $u\in \tilde{\mathbf{U}}_{1+\delta}(M)$, we have
\[|\varphi(a)-\varphi(u^*au)|\le \eta.\]
Since $\varphi$ is linear and $M_+$ spans $M$, we may assume that $a\in M_+$ and that 
$\max(\|a\|_{\varphi},\|a^*\|_{\varphi})=1$.
Choose $\delta>0$ such that $\delta(4+7\delta+4\delta^2+\delta^3)\le \eta$. Then for every $u\in \tilde{\mathbf{U}}_{1+\delta}(M)$, we have 
$\varphi(u^*au)\le (1+\delta)^2\varphi(a)$, so that 
$$\varphi(u^*au)-\varphi(a)\le (2\delta+\delta^2)\varphi(a)\le (2\delta+\delta^2)\le \eta.$$ On the other hand, $\varphi(uu^*auu^*)\le (1+\delta)^2\varphi(u^*au)$, thus $$\varphi(uu^*auu^*)-\varphi(u^*au)\le (2\delta+\delta^2)\varphi(u^*au)\le (1+\delta)^2(2\delta+\delta^2)\varphi(a)$$ and 
\begin{align*}
    |\varphi(uu^*auu^*)-\varphi(a)|&\le |\varphi((uu^*-1)auu^*)|+|\varphi(a(uu^*-1))|\\
    &\le \|uu^*-1\|_{\varphi}\|auu^*\|_{\varphi}+\|a^*\|_{\varphi}\|uu^*-1\|_{\varphi}\\
    &\le \delta(1+\delta)^2\|a\|_{\varphi}+\|a^*\|_{\varphi}\delta\\
    &\le \delta(2+2\delta+\delta^2).
\end{align*}
This implies that 
\begin{align*}
    \varphi(a)-\varphi(uau^*)&\le |\varphi(a)-\varphi(uu^*auu^*)|+\varphi(uu^*auu^*)-\varphi(uau^*)\\
    &\le \delta(2+2\delta+\delta^2)+\delta(1+\delta)^2(2+\delta)\\
    &=\delta(4+7\delta+4\delta^2+\delta^3)\le \eta.
\end{align*}
This shows that $|\varphi(u^*au)-\varphi(a)|\le \eta$. Consequently, for every $\eta>0$ we have  
$$|\varphi(y)-\varphi(a)|\le \eta \text{ for }y\in \widetilde{C}_{\varphi}(a,\delta).$$

Therefore, $\varphi(x)=\varphi(a)$ holds for every $x\in \widetilde{C}_{\varphi}(a)$. 
Now assume that (b) holds. Since the sets $\widetilde{C}_{\varphi}(a,\delta)\cap \mathbb{C}1$ form a decreasing family of nonempty $\sigma$-weakly compact sets, its intersection $\widetilde{C}_{\varphi}(a)\cap \mathbb{C}1$ is nonempty. Let $\lambda\in \mathbb{C}$ be such that $\lambda 1\in \widetilde{C}_{\varphi}(a)$. Then by the above argument, we have $\lambda=\varphi(\lambda 1)=\varphi(a)$. This shows (c).\\
To prove that (c)$\implies$(a), take $a\in \BC(M,\varphi)$ and $\ep>0$. By Lemma \ref{lem tildeepsilon and BC}, there exists $\delta>0$ such that $\tilde{\ep}_{M}(a,\delta)<\frac{\ep}{2}$. Since $\varphi(a)1\in \widetilde{C}_{\varphi}(a,\delta)$, there exist $u_1,\dots,u_N\in \tilde{\mathbf{U}}_{1+\delta}(M)$ such that 
$$\left \|\frac{1}{N}\sum_{k=1}^Nu_k^*au_k-\varphi (a)1\right \|_{\varphi}<\frac{\ep}{2}.$$
On the other hand, for each $k=1,\dots,N$, we have 
$\|u_k^*au_k-a\|_{\varphi}\le \tilde{\ep}_M(a,\delta)<\frac{\ep}{2}$. Thus 
\begin{align*}
    \|a-\varphi(a)1\|_{\varphi}&\le \left \|\frac{1}{N}\sum_{k=1}^N(a-u_k^*au_k)\right \|_{\varphi}+\left \|\frac{1}{N}\sum_{k=1}^Nu_k^*au_k-\varphi(a)1\right \|_{\varphi}\\
    &< \frac{1}{N}\sum_{k=1}^N\|a-u_k^*au_k\|_{\varphi}+\frac{\ep}{2}<\ep.
\end{align*}
Since $\ep>0$ is arbitrary, $a=\varphi(a)1$ holds. 
\end{proof}

We conclude this subsection by pointing to an asymptotic connection between almost centralizing unitaries and elements of $\tilde{\mathbf{U}}_{1+\delta}$:

\begin{prop}\label{almostcentralizingalmostunitary}
    For all $\epsilon,\eta>0$, there is $\delta=\delta(\epsilon,\eta)>0$ such that whenever $(M,\varphi)$ is a W$^*$-probability space and $u\in U(M)$ is such that $\|u\varphi-\varphi u\|<\delta$, then there is a contraction $v\in \tilde{\mathbf{U}}_{1+\eta}(M,\varphi)$ with $\|u-v\|_\varphi^\#<\epsilon$. 
\end{prop}

\begin{proof}
Suppose this is not the case for some $\epsilon,\eta>0$.  
Then for each $n\geq 1$, there is a W$^*$-probability space $(M_n,\varphi_n)$ and $u_n\in U(M_n)$ such that $\|u_n\varphi_n-\varphi_n u_n\|<1/n$ and yet $d(u_n,\tilde{\mathbf{U}}_{1+\eta}(M_n,\varphi_n))\geq \epsilon$.  Set $(M,\varphi):=\prod_\u (M_n,\varphi_n)$.  By Lemma \ref{lem asymptotic centralizer}, we may consider $u:=(u_n)_\u\in M_\varphi$.  By Lemma \ref{lem representing u by 1+deltabdd}, we may write $u=(x_n)_\u$ with each $x_n\in \tilde{\mathbf{U}}_{1+\eta}$ a contraction.  Since $\|u_n-x_n\|_\varphi^\#<\epsilon$ for $\u$-almost all $n$, this is a contradiction.
\end{proof}

The previous proposition has a sort of converse:

\begin{prop}\label{almmostunitaryalmostcentralizing}
For every $\epsilon,\delta>0$, there is $\eta>0$ such that whenever $(M,\varphi)$ is a W$^*$-probability space and $v\in \tilde{\mathbf{U}}_{1+\eta}(M)$ is a contraction, then there is $u\in U(M)$ such that $\|u-v\|_\varphi^\#<\epsilon$ and $\|u\varphi-\varphi u\|<\delta$.
\end{prop}

\begin{proof}
Suppose, towards a contradiction, that the proposition is false for some $\epsilon,\delta>0$.  Consequently, for every $n\geq 1$, there are W$^*$-probability spaces $(M_n,\varphi_n)$ and contractions $v_n\in \tilde{\mathbf{U}}_{1+1/n}(M_n)$ for which, whenever $u_n$ is a unitary in $M_n$ with $\|u_n-v_n\|_\varphi^\#<\epsilon$, one has $\|u_n\varphi_n-\varphi_n u_n\|\geq \delta$.  Take a nonprincipal ultrafilter $\u$ on $\bb N$ and set $(M,\varphi):=\prod_\u (M_n,\varphi_n)$ and $v:=(v_n)_\u$.  Then $v$ is a totally 1-bounded unitary in $M$ and thus $v\in M_\varphi$ by Corollary \ref{totallyoneboundedunitary}.  Write $v=(u_n)_\u$ with each $u_n$ a unitary in $M_n$; note that $\|u_n-v_n\|_\varphi^\#<\epsilon$ for $\u$-almost all $n$.  Since $v\varphi=\varphi v$, by Fact \ref{ultralimitbimodule} we have that $$\lim_\u \|u_n\varphi_n-\varphi_n u_n\|=0,$$ whence $\|u_n\varphi_n-\varphi_n u_n\|<\delta$ for $\u$-almost all $n$, contradicting the choice of $v_n$.
\end{proof}

\section{Axiomatizing W$^*$-probability spaces with trivial bicentralizer}\label{sec axiomatize TBC}

\subsection{On selfless W$^*$-probability spaces}

In this subsection we prove that the class of selfless W$^*$-probability spaces is axiomatizable.  First, we recall the relevant definitions from the introduction.

An inclusion $(N,\psi)\subseteq (M,\varphi)$ is called \textbf{existential} if:  for any quantifier-free formula $\theta(\vec x,\vec a)$ with parameters from $N$, we have
$$\left(\inf_{\vec x}\theta(\vec x,\vec a)\right)^{(N,\psi)}=\left(\inf_{\vec x} \theta(\vec x,\vec a)\right)^{(M,\varphi)}.$$
Equivalently, the inclusion $(N,\psi)\subseteq (M,\varphi)$ is existential if there is an embedding $(M,\varphi)\hookrightarrow (N,\psi)^\u$ that restricts to the diagonal embedding $(N,\psi)\hookrightarrow (N,\psi)^\u$.  If $N$ is separable, then $\u$ can be taken to be any nonprincipal ultrafilter on any index set; otherwise, one needs to choose $\u$ to be a suitably good ultrafilter (see Subsection \ref{good}).  Identifying a W$^*$-probability space with its image under an embedding, one may also speak of existential embeddings.  

We will need the following well-known and easy facts about existential embeddings:

\begin{fact}\label{ecembeddingfact1}
Suppose that $i:(N,\psi)\hookrightarrow (M,\varphi)$ and $j:(M,\varphi)\hookrightarrow (P,\rho)$ are embeddings.
\begin{enumerate}
    \item  If both $i$ and $j$ are existential, then the composition $j\circ i:(N,\psi)\hookrightarrow (P,\rho)$ is also existential.
    \item If the composition $j\circ i:(N,\psi)\hookrightarrow (P,\rho)$ is existential, then so is $i$.
\end{enumerate}
\end{fact}

\begin{fact}\label{ecembeddingfact2}
Suppose that, for each $i\in I$, $j_i:(N_i,\psi_i)\hookrightarrow (M_i,\varphi_i)$ are existential embeddings.  Suppose further that $\u$ is an ultrafilter on $I$.  Then the ultraproduct map $\prod_\u j_i:\prod_\u (N_i,\psi_i)\hookrightarrow \prod_\u (M_i,\varphi_i)$ is also existential. 
\end{fact}

The W$^*$-probability space $(M,\varphi)$ is called \textbf{selfless} if the first factor inclusion $(M,\varphi)\subseteq (M,\varphi)*(M,\varphi)$ is existential.

Before proving our axiomatizability result, we need a few lemmas.

\begin{lem}\label{inducedembedding}
Suppose that $(M_{1,i},\varphi_{1,i})_{i\in I}$ and $(M_{2,i},\varphi_{2,i})_{i\in I}$ are families of W$^*$-probability spaces and $\u$ is an ultrafilter on $I$. 
Let $(M_k^{\cU},\varphi_k^{\cU})=\prod_{\cU}(M_{k,i},\varphi_{k,i})$ and $(M^{\cU},\varphi^{\cU})=\prod_{\cU}(M_i,\varphi_i)$ denote the corresponding Ocneanu ultraproducts, where $$(M_i,\varphi_i)=(M_{1,i},\varphi_{1,i})*(M_{2,i},\varphi_{2,i}), \quad i\in I,\,k=1,2.$$ 
Then the canonical embeddings of $M_1^\mathcal{U}$ and $M_2^\mathcal{U}$ into $M^\mathcal{U}$ are freely independent with respect to the ultraproduct state $\varphi^\mathcal{U}$. 

Consequently, there is a canonical embedding of W$^*$-probability spaces 
$$(\prod_\u (M_{1,i},\varphi_{1,i}))*(\prod_\u (M_{2,i},\varphi_{2,i}))\hookrightarrow \prod_\u ((M_{1,i},\varphi_{1,i})*(M_{2,i},\varphi_{2,i})).$$
\end{lem}

\begin{proof}
The inclusions of W$^*$-probability spaces $$(M_{k,i},\varphi_{k,i}) \subset (M_i,\varphi_i)\quad i\in I,\,k=1,2$$ induce an inclusion of the W$^*$-probability spaces $$(M_k^\mathcal{U},\varphi_k^{\cU}) \subset (M^\mathcal{U},\varphi^{\cU})\quad k=1,2.$$ Here, we identify the ultrapowers $M_k^\mathcal{U}$ with their respective images in $M^\mathcal{U}$.

To verify freeness, let $x_1, x_2, \dots, x_m$ be an alternating sequence of elements such that $x_j \in M_{\ell_j}^\mathcal{U}$ and $\varphi_{\ell_j}^\u(x_j) = 0$ for all $j=1, \dots, m$ ($\ell_j\neq \ell_{j+1},\,j=1,\dots,m-1$). 

Each $x_j$ can be represented by a norm-bounded sequence $(x_{j, i})_{i \in I}$ with $x_{j, i} \in M_{\ell_j,i}$. The condition $\varphi_{\ell_j}^\u(x_j) = 0$ implies that $\lim_{i\to \cU} \varphi_{\ell_j,i}(x_{j, i}) = 0$. By replacing each representative $x_{j, i}$ with $x_{j, i} - \varphi_{\ell_j,i}(x_{j, i})1$, we can choose representatives such that $\varphi_{\ell_j,i}(x_{j, i}) = 0$ holds exactly, for every $i\in I$. 

Thus, for each fixed $i\in I$, the elements $x_{1, i}, \dots, x_{m, i}$ form an alternating word of centered elements in the original free product $(M_{1,i},\varphi_{1,i}) * (M_{2,i},\varphi_{2,i})$. By the free independence of $M_{1,i}$ and $M_{2,i}$ inside $M_i$ with respect to $\varphi_i=\varphi_{1,i}*\varphi_{2,i}$, we have
\[
\varphi_i(x_{1, i} x_{2, i} \cdots x_{m, i}) = 0 \quad \text{for all } i\in I.
\]
Passing to the ultralimit along $\mathcal{U}$, we obtain
\[
\varphi^{\u}(x_1 x_2 \cdots x_m) = \lim_{i \to \mathcal{U}} \varphi_i(x_{1, i} x_{2, i} \cdots x_{m, i}) = 0.
\]
This shows that the subalgebras $M_1^\mathcal{U}$ and $M_2^\mathcal{U}$ are free with respect to the state $\varphi^\mathcal{U}$. By the universal property of the free product, the von Neumann algebra they generate inside $M^\mathcal{U}$ is canonically isomorphic to $(M_1^\mathcal{U}, \varphi_1^\mathcal{U}) * (M_2^\mathcal{U}, \varphi_2^\mathcal{U})$, which yields the desired canonical embedding.
\end{proof}

\begin{lem}\label{freeexistential}
A free product of existential embeddings of W$^*$-probability spaces is again existential.   
\end{lem}

\begin{proof}
For $i=1,2$, let $\alpha_i:(M_i,\varphi_i)\hookrightarrow (N_i,\psi_i)$ be existential embeddings.  Lemma \ref{inducedembedding} shows that $\alpha_1*\alpha_2$ yields an embedding \begin{alignat}{2}
(M_1,\varphi_1)*(M_2,\varphi_2)&\hookrightarrow (N_1,\psi_1)*(N_2,\psi_2)\notag \\ \notag
    &\hookrightarrow (M_1,\varphi_1)^\u*(M_2,\varphi_2)^\u\\ \notag 
    &\hookrightarrow ((M_1,\varphi_1)*(M_2,\varphi_2))^\u,
\end{alignat}
as desired.
\end{proof}

\begin{lem}\label{selflessexistential}
If $(N,\psi)$ is an existential substructure of $(M,\varphi)$ and $(M,\varphi)$ is selfless, then $(N,\psi)$ is selfless.
\end{lem}

\begin{proof}
By Lemma \ref{freeexistential}, we have that the inclusion $(N,\psi)*(N,\psi)\subseteq (M,\varphi)*(M,\varphi)$ is existential.  Since the inclusion $(M,\varphi)\subseteq (M,\varphi)*(M,\varphi)$ is also existential, the result follows.  
\end{proof}

\begin{remark}
One can prove Lemma \ref{selflessexistential} without using Lemma \ref{freeexistential} as follows.
Since both $(N,\psi)\subseteq (M,\varphi)$ and 
$(M,\varphi)\subseteq (M,\varphi)*(M,\varphi)$ are existential, there are ultrafilters $\cU_1,\cU_2$ on some index sets $I_1, I_2$ and embeddings of W$^*$-probability spaces of the form
\[(N,\psi)\stackrel{\pi}{\subseteq}(M,\varphi)\stackrel{\theta}{\subseteq}(N,\psi)^{\cU_1}\]
and 
\[(M,\varphi)\stackrel{j_1}{\subseteq}(M,\varphi)*(M,\varphi)\stackrel{k_1}{\subseteq}(M,\varphi)^{\cU_2},\]
so that $\theta\circ \pi=\iota_N$ and $k_1\circ j_1=\iota_M$; here $\iota_N\colon (N,\psi)\hookrightarrow (N,\psi)^{\cU_1}$ and $\iota_M\colon (M,\varphi)\hookrightarrow (M,\varphi)^{\cU_2}$ are the corresponding diagonal embeddings, and $j_1$ is the first factor embedding. 
Then we have the following embedding: 
\begin{align*}
    (N,\psi)&\stackrel{i_1}{\subseteq}(N,\psi)*(N,\psi)\stackrel{\pi*\pi}{\subseteq}(M,\varphi)*(M,\varphi)\\
    &\stackrel{k_1}{\subseteq}(M,\varphi)^{\cU_2}
    \stackrel{\theta^{\cU_2}}{\subseteq}((N,\psi)^{\cU_1})^{\cU_2}=(N,\psi)^{\cU_2\otimes \cU_1}.
\end{align*}
Here, $i_1$ is the first factor embedding. Since $(k_1\circ j_1)(y)=y$ for all $y\in M$, for each $x\in N$, we have 
\begin{align*}
    (\theta^{\u_2}\circ k_1\circ (\pi*\pi)\circ i_1)(x)&=(\theta^{\u_2}\circ k_1\circ j_1)(\pi(x))\\ &=\theta^{\u_2}(\pi(x))\\
    &=x. 
\end{align*}
Therefore, this shows that $(N,\psi)\subseteq (N,\psi)*(N,\psi)$ is existential.
\end{remark}

\begin{thm}\label{selflessaxiomatizable}
The class of selfless W$^*$-probability spaces is axiomatizable.
\end{thm}

\begin{proof}
We verify that the class of selfless W$^*$-probability spaces is closed under isomorphism, ultraproduct, and ultraroots.  Closure under isomorphism is clear and closure under ultraroot is a special case of Lemma \ref{selflessexistential}.  To prove closure under ultraproducts, note that Lemma \ref{inducedembedding} yields
$$\prod_\u (M_i,\varphi_i)\subseteq \prod_\u (M_i,\varphi_i)*\prod_\u (M_i,\varphi_i)\hookrightarrow \prod_\u ((M_i,\varphi_i)*(M_i,\varphi_i))$$
whose composition is the ultraproduct of the first factor embeddings.  Since the ultraproduct of existential embeddings is existential by Fact \ref{ecembeddingfact2}, the composition is existential, whence so is the first inclusion by Fact \ref{ecembeddingfact1}(2), yielding the desired result.
\end{proof}

\begin{remark}
A result similar to the above was proven in the setting of \cstar-probability spaces by Robert \cite[Theorem 4.5]{robert2025selfless}:  the class of selfless \cstar-probability spaces $(A,\rho)$ with $A$ a simple, purely infinite \cstar-algebra, is axiomatizable.  It is in fact $\forall\exists$-axiomatizable, which is also true in our context, as we will soon prove.
\end{remark}

We next observe that the union of a chain of selfless W$^*$-probability spaces is again selfless.  To prove this, we need the following:

\begin{lem}\label{injectivelemma}
Suppose that $(M_i,\varphi_i)_{i\in I}$ is a directed system of W$^*$-probability spaces with injective bonding maps and direct limit $(M,\varphi)$.  Further suppose that there is a W$^*$-probability space $(N,\psi)$ and embeddings $\alpha_i:(M_i,\varphi_i)\hookrightarrow (N,\psi)$ compatible with the directed system.  Then the induced map $\alpha:(M,\varphi)\to (N,\psi)$ is also injective.
\end{lem}

\begin{proof}
Without loss of generality, let us assume that each $(M_i,\varphi_i)\subseteq (M,\varphi)$ and set $M_0:=\bigcup_{i\in I}M_i$, a dense $*$-subalgebra of $M$.  For each $a\in M_0$, set $$X_a:=\{i\in I \ : \ a\in M_i\}.$$  Since $I$ is directed, the family $(X_a)_{a\in M}$ has the finite intersection property.  Let $\u$ be an ultrafilter on $I$ such that each $X_a\in \u$.  Define a map $$\Phi:(M,\varphi)\to \prod_\u (M_i,\varphi_i)$$ by $\Phi(a)=(E_i(a))_\u$, where $E_i:M\to M_i$ is a state-preserving conditional expectation map.  Note that, for all $a\in M_0$ and $\u$-almost all $i\in I$, we have $E_i(a)=a$.  It follows that $\Phi$ is an embedding of W$^*$-probability spaces and hence $\prod_\u \alpha_i\circ \Phi:M\to N^\u$ is an embedding.  But $\Phi(a)=(\alpha(a))_\u$ for all $a\in M_0$, whence $\alpha$ is also injective, as desired.
\end{proof}

\begin{prop}\label{limit}
Suppose that $(M_i,\varphi_i)_{i\in I}$ is a directed system of W$^*$-probability spaces with embeddings as bonding maps and direct limit $(M,\varphi)$.  If each $(M_i,\varphi_i)$ is selfless, then so is $(M,\varphi)$.
\end{prop}

\begin{proof}
We may assume, after identifying each $(M_i,\varphi_i)$ with its image in the direct limit, that
\[
        M_i\subseteq M_j\subseteq M
        \quad\text{whenever }i\leq j,
        \qquad
        \varphi_j|_{M_i}=\varphi_i,
        \qquad
        \varphi|_{M_i}=\varphi_i,
\]
and that $\bigcup_{i\in I}M_i$ is dense in $M$.

Set
$$A_i:=(M_i,\varphi_i)*(M_i,\varphi_i) \text{ and }
        A:=(M,\varphi)*(M,\varphi).$$
The inclusions $M_i\subseteq M_j$ induce state-preserving embeddings $A_i\hookrightarrow A_j$
and the inclusions $M_i\subseteq M$ induce state-preserving embeddings $A_i\hookrightarrow A$.

We prove that the first factor embedding
\[
        (M,\varphi)\hookrightarrow (M,\varphi)*(M,\varphi)=A
\]
is existential. Let $\theta(\vec x,\vec a)$ be a quantifier-free formula with parameters $\vec a$ from $M$.  Fix $\epsilon>0$. Choose a tuple $\vec b$ from $A$ such that
\[
        \theta^A(\vec b,\vec a)
        <
        \inf_{\vec x\in A}\theta^A(\vec x,\vec a)+\epsilon.
\]
By uniform continuity of the quantifier-free formula $\theta$, choose $\eta>0$ such that replacing the tuple of variables and parameters by tuples within $\eta$ changes the value of $\theta$ by less than $\epsilon$.

Since $\bigcup_i M_i$ is dense in $M$, there are some $i_0\in I$ and a tuple $\vec a_0$ from $M_{i_0}$ such that $d_M(\vec a,\vec a_0)<\eta$.
After increasing $i_0$ if necessary, there are some $j\geq i_0$ and a tuple $\vec b_j$ from $A_j$
such that $d_A(\vec b,\vec b_j)<\eta$.
Then
\[
        \bigl|\theta^A(\vec b,\vec a)-\theta^{A_j}(\vec b_j,\vec a_0)\bigr|<\epsilon,
\]
where we identify $A_j$ with its image in $A$.

Since $(M_j,\varphi_j)$ is selfless, the first factor embedding $(M_j,\varphi_j)\hookrightarrow A_j$
is existential. 
Thus we may choose a tuple $\vec c$ from $M_j$ such that
\[
        \theta^{M_j}(\vec c,\vec a_0)
        <
        \theta^{A_j}(\vec b_j,\vec a_0)+\epsilon.
\]
Viewing $\vec c$ as a tuple from $M$, and using again uniform continuity, we obtain
\[
        \theta^M(\vec c,\vec a)
        <
        \theta^{M_j}(\vec c,\vec a_0)+\epsilon.
\]
Combining the preceding inequalities gives
\[
\begin{aligned}
        \inf_{\vec x\in M}\theta^M(\vec x,\vec a)
        &\leq \theta^M(\vec c,\vec a) \\
        &< \theta^{M_j}(\vec c,\vec a_0)+\epsilon \\
        &< \theta^{A_j}(\vec b_j,\vec a_0)+2\epsilon \\
        &< \theta^A(\vec b,\vec a)+3\epsilon \\
        &< \inf_{\vec x\in A}\theta^A(\vec x,\vec a)+4\epsilon.
\end{aligned}
\]
Since $\epsilon>0$ was arbitrary, this proves
\[
        \inf_{\vec x\in M}\theta^M(\vec x,\vec a)
        =
        \inf_{\vec x\in A}\theta^A(\vec x,\vec a).
\]
Therefore the first factor embedding $(M,\varphi)\hookrightarrow (M,\varphi)*(M,\varphi)$
is existential and hence $(M,\varphi)$ is selfless.
\end{proof}

It is a well-known fact in model theory (see, for example, \cite[Proposition 2.4.4]{farah2021model}), that an axiomatizable class of structures is $\forall\exists$-axiomatizable if and only if it is closed under unions of chains.  As a result of the previous proposition, we have:

\begin{cor}\label{selflessAEaxiomatizable}
The class of selfless W$^*$-probability spaces is $\forall\exists$-axiomatizable.
\end{cor}

\begin{remark}
An alternate proof of the previous corollary would be to use the fact that existential substructures of selfless W$^*$-probability spaces are again selfless (Lemma \ref{selflessexistential} above) together with Keisler's sandwich theorem (see \cite[Theorem 4.2]{goldbringhoudayer}).  This then gives an alternate proof of Proposition \ref{limit}.
\end{remark}

We end this section with a few further observations about selfless W$^*$-probability spaces.

\begin{lem}\label{separableenoughselfless}
A W$^*$-probability space $(M,\varphi)$ is selfless if and only if all of its separable existential substructures are selfless.
\end{lem}

\begin{proof}
The forward direction follows from Lemma \ref{selflessexistential} and the reverse direction follows from Lemma \ref{limit} and Downward L\"owenheim-Skolem (as $(M,\varphi)$ is the direct limit of all of its separable elementary substructures).
\end{proof}

\begin{prop}\label{selflessfactor}
A nontrivial selfless W$^*$-probability space is a factor of type II$_1$ or type III$_\lambda$ for some $\lambda\in (0,1]$.
\end{prop}

\begin{proof}
Suppose that $(M,\varphi)$ is a nontrivial selfless W$^*$-probability space.  The proof of the implication (ii)$\Rightarrow$(iii) in \cite[Theorem A]{cyrilamine} then shows that there is a diffuse W$^*$-probability space $(N,\psi)$ such that $(M,\varphi)\subseteq (M,\varphi)*(N,\psi)$ is existential; we stress that even though $(M,\varphi)$ is assumed to be both separable and diffuse in the statement of \cite[Theorem A]{cyrilamine}, neither of these assumptions are used to deduce the aforementioned fact.  By \cite[Theorem 3.4]{ueda2011factoriality}, $(M,\varphi)*(N,\psi)$ is a factor.  By \cite[Lemma 3.3]{goldbringhoudayer}, $M$ is a factor.

We next show that $M$ cannot be type I$_n$.  If $M\cong M_n(\bb C)$, then $(M_n(\bb C),\varphi)^\u\cong (M_n(\bb C),\varphi)$ by \cite[Lemma 2.8]{MT16Rohlin} and thus cannot contain a copy of the infinite-dimensional W$^*$-probability space $(M_n(\bb C),\varphi)*(M_n(\bb C),\varphi)$, whence $M_n(\bb C)$ is not selfless.

Finally, as noted in Fact \ref{bc for generalfactors}, a \emph{separable} selfless W$^*$-probability space cannot be type II$_\infty$ or III$_0$.  By Lemma \ref{selflessexistential} and Downward L\"owenheim-Skolem, it thus suffices to check that an elementary substructure of a II$_\infty$ factor (resp. III$_0$ factor) is also a type II$_\infty$ factor (resp. III$_0$ factor).  This follows from the fact that types II$_1$, II$_\infty$, and III$_\lambda$ for $\lambda\in (0,1]$ are preserved under elementary equivalence (see \cite[Propositions 8.7 and 8.8]{AGHS25}).  
\end{proof}

\begin{remark}
  The main result of \cite[Theorem A]{cyrilamine} is that for separable, diffuse W$^*$-probability spaces $(M,\varphi)$, being selfless is equivalent to $\BC(M,\varphi)=\bb C$.  We will soon remove the separability assumption in their result, but we comment that the diffuse assumption in their result is necessary by the previous lemma.  Moreover, to deduce that selfless W$^*$-probability spaces are factors using their theorem, one needs to first assume that the underlying von Neumann algebra is diffuse to then conclude that it has trivial bicentralizer, from which factoriality follows immediately; in the proof of the previous proposition, we first deduced factoriality directly from the selflessness assumption, and then diffuseness followed quite easily.
\end{remark}

\subsection{Selflessness and trivial bicentralizer}

In \cite{cyrilamine}, it is shown that, for a \emph{separable} diffuse W$^*$-probability space, one has that selflessness is equivalent to having trivial bicentralizer.  We wish to extend this fact to all diffuse W$^*$-probability spaces.  First, we need some preparation. 

\begin{prop}\label{upanddown}
Fix an inclusion $(N,\psi)\subseteq (M,\varphi)$ of W$^*$-probability spaces.  Then for all $a\in N$ and $\epsilon,\delta>0$, we have:
\begin{enumerate}
    \item If $\epsilon_M(a,\delta)<\epsilon$, then $\epsilon_N(a,\delta)\leq \epsilon$.
    \item If $(N,\psi)$ is an existential substructure of $(M,\varphi)$ and $\epsilon_N(a,\delta)<\epsilon$, then $\epsilon_M(a,\delta/2)\leq \epsilon$.
\end{enumerate}
\end{prop}

\begin{proof}
For (1), suppose that $\epsilon_M(a,\delta)<\epsilon$.  Fix $u\in U(N)$ such that $\|u\psi-\psi u\|\leq \delta$.  Let $E:M\to N$ be a state-preserving conditional expectation.  Then for $b\in \operatorname{ball} M$, one has $|\varphi(ub)-\varphi(bu)|=|\psi(uE(b))-\psi(E(b)u)|\leq \delta$.  It follows that $\|u\varphi-\varphi u\|\leq \delta$ and thus $\|u^*au-a\|_\varphi<\epsilon$; since $u\in U(N)$ was an arbitrary unitary with $\|u\psi-\psi u\|\leq \delta$, we have $\epsilon_N(a,\delta)\leq \epsilon$.

For (2), suppose that $\epsilon_N(a,\delta)<\epsilon$ and fix $u\in U(M)$ with $\|u\varphi-\varphi u\|\leq \delta/2$.  Fix an embedding $i:(M,\varphi)\hookrightarrow (N,\psi)^\u$ that restricts to the diagonal embedding on $(N,\psi)$ and write $i(u)=(u_i)_\u$, where $u_i\in \mathbf{U}(N)$ for every $i\in I$.  Then $\|i(u)\psi^\u-\psi^\u i(u)\|\leq \delta/2$ by the argument in part (1).  Therefore, by Fact \ref{ultralimitbimodule}, we have $\lim_\u \|u_i\psi-\psi u_i\|<\delta$ and so $\|u_i\psi-\psi u_i\|<\delta$ for $\u$-almost all $i$.  It follows that $\|u_i^*au_i-a\|_{\varphi_i}<\epsilon$ for $\u$-almost all $i$ and thus $\|u^*au-a\|_\varphi\leq \epsilon$.
\end{proof}

\begin{cor}\label{BCpreceq}
If $(N,\psi)$ is an existential substructure of $(M,\varphi)$, then $\BC(N,\psi)=\BC(M,\varphi)\cap N$.
\end{cor}

\begin{remark}\label{rem separableunnecessary}
By Lemma \ref{BCpreceq}, one sees now that if the bicentralizer problem has a positive solution, then $\BC(M,\varphi)=\bb C$, regardless of whether or not $M$ is separable.
\end{remark}

We can now prove the bicentralizer analog of Lemma \ref{separableenoughselfless}.
\begin{lem}\label{separableenoughbicentralizer}
If $(M,\varphi)$ is a W$^*$-probability space, then $\BC(M,\varphi)=\mathbb C$ if and only if $\BC(N,\psi)=\C$ for all separable existential substructures $(N,\psi)$ of $(M,\varphi)$.
\end{lem}

\begin{proof}
This is an immediate application of Corollary \ref{BCpreceq} and Downward L\"owenheim-Skolem.
\end{proof}

\begin{lem}\label{diffuselemma}
If $(N,\psi)$ is an existential substructure of a diffuse W$^*$-probability space $(M,\varphi)$, then $(N,\psi)$ is also diffuse.
\end{lem}

\begin{proof}
Let $p$ be a nonzero projection in $N$.  Let $q$ be a nonzero projection in $M$ with $q<p$.  Fix an embedding $i\colon (M,\varphi)\hookrightarrow (N,\psi)^{\cU}$ that restricts to the diagonal embedding on $(N,\psi)$ and write $q=(q_n)_\u$ with each $q_n\in N$ a subprojection of $p$.  Then for $\u$-almost all $n$, $q_n$ is a nonzero proper subprojection of $p$; since $p$ was an arbitrary nonzero projection in $N$, we have that $N$ is diffuse.   
\end{proof}

We can now remove the separability requirement in the result of Houdayer and Marrakchi:

\begin{thm}\label{generalizing}
For any diffuse W$^*$-probability space $(M,\varphi)$ (not necessarily separable), we have that $(M,\varphi)$ is selfless if and only if $\BC(M,\varphi)=\mathbb C$.
\end{thm}

\begin{proof}
The theorem follows immediately from the main result of \cite{cyrilamine} using Lemmas \ref{separableenoughselfless}, \ref{separableenoughbicentralizer}, and \ref{diffuselemma}. 
\end{proof}

Theorem \ref{selflessaxiomatizable}, Corollary \ref{selflessAEaxiomatizable}, and Theorem \ref{generalizing} immediately yield:

\begin{cor}\label{bicentralizeraxiomatizable}
For any axiomatizable class $\cal K$ of diffuse W$^*$-probability spaces, the set of elements of $\cal K$ with trivial bicentralizer is axiomatizable.  If $\cal K$ is $\forall\exists$-axiomatizable, then so is the class of elements of $\cal K$ with trivial bicentralizer.  In particular, the class of III$_1$ factors with trivial bicentralizer is $\forall\exists$-axiomatizable.
\end{cor}

Similarly, Lemma \ref{limit} and Theorem \ref{generalizing} yield:

\begin{cor}
The class of diffuse W$^*$-probability spaces with trivial bicentralizer  is closed under direct limits.
\end{cor} 

\subsection{Pseudo-periodic III$_1$ factors}
By Proposition \ref{typeofultraproduct}, if $(M_i,\varphi_i)_{i\in I}$ is a family of W$^*$-probability spaces, where $M_i$ is a type III$_{\lambda_i}$-factor with $\lambda_i\in (0,1)$, 
and $\u$ is an ultrafilter on $I$ for which $\lim_\u \lambda_i=1$, if we set $(M,\varphi):=\prod_\u (M_i,\varphi_i)$, then $M$ is a type III$_1$ factor, as is then any W$^*$-probability space elementarily equivalent to $(M,\varphi)$.  We can give an alternate characterization of such W$^*$-probability spaces:  

\begin{prop}\label{pseudoperiodicequivalence}
Suppose that $(M,\varphi)$ is a W$^*$-probability space. Then the following conditions are equivalent:
\begin{enumerate}
    \item There is a family $(M_i,\varphi_i)_{i\in I}$ of W$^*$-probability spaces, where $M_i$ is a type III$_{\lambda_i}$-factor with $\lambda_i\in (0,1)$, and an ultrafilter $\u$ on $I$ such that $\lim_\u \lambda_i=1$ for which $(M,\varphi)\equiv \prod_\u (M_i,\varphi_i)$.
    \item For every sentence $\sigma$ with the property that, for some $\lambda_0\in (0,1)$, we have $\sigma^{(N,\psi)}=0$ whenever $N$ is a III$_\lambda$ factor with $\lambda\in [\lambda_0,1)$, then $\sigma^{(M,\varphi)}=0$.
    \item For every sentence $\sigma$ such that $\sigma^{(M,\varphi)}=0$, every $\epsilon>0$, and every $\lambda_0\in (0,1)$, there is a W$^*$-probability space $(N,\psi)$ with $N$ a type III$_\lambda$-factor with $\lambda\in [\lambda_0,1)$ such that $\sigma^{(N,\psi)}<\epsilon$.
\end{enumerate}
\end{prop}

\begin{proof}
The implications (1) implies (2) and (2) implies (3) are immediate.  Now suppose that $(M,\varphi)$ satisfies (3).  Let $X$ denote the set of sentences $\sigma$ such that $\sigma^{(M,\varphi)}=0$.  Let $I$ denote the set of finite subsets of $X\sqcup \bb N^{>0}$ that intersect both $X$ and $\bb N$.  For each $\sigma\in X$ and $m\geq 1$, let $O_{\sigma,m}$ denote those $i\in I$ that contain both $\sigma$ and $m$.  Note that the collection of such sets $O_{\sigma,m}$ has the finite intersection property, whence there is an ultrafilter $\u$ on $I$ containing each $O_{\sigma,m}$.  By (3), for each $i\in I$, setting $m_i:=\max (\bb N\cap i)$, we may find a W$^*$-probability space $(M_i,\varphi_i)$ with $M_i$ a III$_{\lambda_i}$ factor satisfying $\lambda_i\geq 1-1/{m_i}$ and with $\sigma^{(M_i,\varphi_i)}<1/m_i$ for all $\sigma\in X\cap i$.  Note then that $\lim_\u \lambda_i=1$ and $M\equiv \prod_\u (M_i,\varphi_i)$, establishing (1).
\end{proof}

Note that any W$^*$-probability space $(M,\varphi)$ satisfying the equivalent conditions of the previous proposition necessarily has that $M$ is a III$_1$ factor, whence the property is independent of the state $\varphi$.  We call such a III$_1$ factor a \textbf{pseudo-periodic III$_1$ factor}. Note that the use of ``period'' here refers to the one for the flow of weights, that is, we do not assume the states appearing in the ultraproduct are periodic. 

Suppose that $M$ is a pseudo-periodic III$_1$ factor and $(M_i,\varphi_i)_{i\in I}$ and $\u$ are as in condition (1) of the previous proposition.  By Proposition \ref{typeofultraproduct}(2), we may assume that each $\varphi_i$ is a $\frac{2\pi}{|\log(\lambda_i)|}$-periodic state on $M_i$.  Then, as mentioned above, $\BC(M_i,\varphi_i)=\bb C$ for each $i\in I$, whence $\BC(M,\varphi)=\bb C$ by Corollary \ref{bicentralizeraxiomatizable} above.  (To apply the corollary, one may assume that each $\lambda\geq 1/2$ and then one can let $\cal K$ be the class of W$^*$-probability spaces whose underlying von Neumann algebra is a III$_\lambda$-factor for $\lambda\geq 1/2$.  This is indeed axiomatizable by \cite[Proposition 8.8]{AGHS25}.)  This establishes the following:

\begin{thm}
If $M$ is a pseudo-periodic III$_1$ factor, then $M$ has trivial bicentralizer.
\end{thm}

In connection with Connes' bicentralizer problem, this then leads to the following:

\begin{question}
Is every III$_1$ factor a pseudo-periodic III$_1$ factor?
\end{question}

A positive solution to the previous question could be viewed as some form of a type III \emph{Lefschetz principle}.  Recall that the classical Lefschetz principle states that an algebraically closed field of characteristic $0$ is elementarily equivalent to an ultraproduct of algebraically closed fields of positive characteristic.  That being said, a question more in line with the classical Lefschetz principle would be:  is every existentially closed III$_1$ factor a pseudo-periodic III$_1$ factor?  However, we already know that all existentially closed III$_1$ factors are selfless and thus have trivial bicentralizer and are more interested in the general form of the question stated above.

Given that existential subfactors of III$_1$ factors with trivial bicentralizer again have trivial bicentralizer, in order to give a positive answer to the bicentralizer problem, it suffices to give a positive answer to the following, a priori easier, question:

\begin{question}
Does every W$^*$-probability space $(M,\varphi)$ with $M$ a III$_1$ factor admit an existential embedding into an ultraproduct $\prod_\u (M_i,\varphi_i)$, where each $M_i$ is a type III$_{\lambda_i}$ factor with $\lambda_i\in (0,1)$? 
\end{question}

There is a connection between the previous question and the \textbf{Effros-Mar\'echal (EM) topology} that we now explain.  Let $(M,\varphi)$ be a W$^*$-probability space, $\SA(M)$ be the set of all von Neumann subalgebras of $M$, and $\SA_{\varphi}(M)$ be the set of all globally $\sigma^{\varphi}$-invariant von Neumann subalgebras of $M$. 
We consider $\SA_\varphi(M)$ as equipped with its EM-topology; for the definition of this topology, see \cite[Section 2]{HW98}. We will only need the following characterization of convergence in the EM-topology, as explained in \cite[page 575]{HW98}: for a sequence $(N_n)_{n=1}^{\infty}$ and another element $N$ in $\SA_{\varphi}(M)$, we have
\[\lim_{n\to \infty}N_n=N\iff \E_{N_n}(x)\xrightarrow{\mathrm{so*}}\E_N(x),\,\,\,x\in M.\]
Here, for each $N\in \SA_{\varphi}(M)$, $\E_N$ is the unique normal faithful $\varphi$-preserving conditional expectation of $M$ onto $N$.
\begin{prop}\label{prop EM implies existential emb}
Let $(M,\varphi)$ be a W$^*$-probability space with $M_*$ separable. Let $(M_n)_{n=1}^{\infty}$ be a sequence in $\SA_{\varphi}(M)$ such that $\displaystyle \lim_{n\to \infty}M_n=M$ in the EM-topology. 
Then for each nonprincipal ultrafilter $\cU$ on $\N$, there exists an existential embedding of $(M,\varphi)$ into $\prod_{\cU}(M_n,\varphi_n)$, where $\varphi_n:=\varphi|_{M_n}$. 
\end{prop}
\begin{proof}
    For each $n$, let $\E_n$ be the unique $\varphi$-preserving conditional expectation of $M$ onto $M_n$. 
    Since $(M_n,\varphi_n)\subseteq (M,\varphi)$ is an inclusion of W$^*$-probability spaces for every $n\in \N$, it induces an embedding of the W$^*$-probability spaces $$\prod_{\cU}(M_n,\varphi_n)\stackrel{k}{\subseteq} \prod_{\cU}(M,\varphi)$$ 
    given by inclusion, and the $\varphi^{\cU}$-preserving conditional expectation is given by $\E((x_n)_{\cU})=(\E_n(x_n))_{\cU}$ for $(x_n)_{\cU}\in \prod_{\cU}(M,\varphi)$. If $x\in M$, then $(\E_n(x))_{n=1}^{\infty}$ defines an element in $\prod_{\cU}(M_n,\varphi_n)$. 
    Define $j\colon M\to \prod_{\cU}(M_n,\varphi_n)$ by $j(x)=(\E_n(x))_{\cU}$. Since 
    $\E_n(x)\xrightarrow{\mathrm{so*}}x$, it is an injective unital $*$-homomorphism, and it is normal and faithful, since $\psi(j(x))=\varphi(x)$, where $\psi=(\varphi_n)_{\cU}$ is a faithful normal state on $\prod_{\cU}(M_n,\varphi_n)$. The $\psi$-preserving conditional expectation $\ep\colon \prod_{\cU}(M_n,\varphi_n)\to M$ is given by $\ep((x_n)_{\cU}):=\mathrm{wot-}\lim_{n\to \cU}x_n$. Thus, we obtain the embeddings of W$^*$-probability spaces 
    \[(M,\varphi)\stackrel{j}{\subseteq} \prod_{\cU}(M_n,\varphi_n)\stackrel{k}{\subseteq} \prod_{\cU}(M,\varphi).\]
    Moreover, for each $x\in M$, $(k\circ j)(x)=(\E_n(x))_{\cU}=x$ in $\prod_{\cU}(M,\varphi)$ as $\E_n(x)\xrightarrow{\mathrm{so*}}x$. Therefore, the embedding $(M,\varphi)\subseteq \prod_{\cU}(M_n,\varphi_n)$ is existential.  
\end{proof}

\begin{cor}
If, in the context of Proposition \ref{prop EM implies existential emb}, we have that each $M_n$ is diffuse and $\BC(M_n,\varphi_n)=\bb C$, then $\BC(M,\varphi)=\bb C$ as well.
\end{cor}

\begin{remark}
We give an example showing that the existential embedding $$(M,\varphi)\hookrightarrow \prod_\u(M_n,\varphi_n)$$ constructed in Proposition \ref{prop EM implies existential emb} need not be an elementary embedding.  By the proof in \cite[Proposition 4.4]{goldbringhoudayer}, for $M=R_{\infty}$ and a suitable faithful normal state $\varphi$ on $M$, there exists an increasing chain of embeddings $$(M_1,\varphi_1)\subseteq (M_2,\varphi_2)\subseteq \dots \subseteq (M,\varphi)=\bigvee_{n\in \N} (M_n,\varphi_n)$$
of $W^*$-probability spaces such that every $M_n$ is a type III$_{\lambda}$ factor for a fixed $0<\lambda<1$, independent of $n$. 
Then $E_{M_n}(x)\xrightarrow{\mathrm{so*}}x$ for every $x\in M$, where $E_{M_n}$ is the unique $\varphi$-preserving conditional expectation of $M$ onto $M_n$. 
Consequently, $\prod_{\cU}(M_n,\varphi_n)$ is also a type III$_{\lambda}$ factor, while $M$ itself is a type III$_1$ factor. Thus, the existential embedding $(M,\varphi)\subseteq \prod_{\cU}(M_n,\varphi_n)$ constructed in Proposition \ref{prop EM implies existential emb} is not an elementary embedding. 
\end{remark}

\subsection{Large centralizers}

The following is a special case of a result of Haagerup, Houdayer, Marrakchi, and the first author (see \cite[Proposition 3.3]{AHHM20}):

\begin{fact}\label{p3.3}
Suppose that $(M,\varphi)$ is a W$^*$-probability space and $\u$ is a nonprincipal ultrafilter on $\mathbb N$.  Then $(M^{\u}_{\varphi^{\u}})'\cap M^{\u}\subseteq \BC(M,\varphi)^{\u}$.   
\end{fact}

In particular, if $\BC(M,\varphi)=\bb C$, then $\varphi^\u$ has large centralizer, that is, $(M^{\u}_{\varphi^\u})'\cap M^\u=\bb C$.  We wish to generalize this latter observation to arbitrary ultraproducts, but under the additional assumption that each state involved has large centralizer:

\begin{prop}\label{ultracyril}
    Let \((M_i,\psi_i)_{i\in I}\) be a family of W\(^*\)-probability spaces such that $(M_i)_{\psi_i}'\cap M_i=\bb C$ for all $i\in I$. Let \(\cU\) be an ultrafilter on \(I\) and set \((M,\psi)=\prod_{\cU}(M_i,\psi_i)\). Then \(M_{\psi}'\cap M=\C.\)  
\end{prop}

The proof of Proposition \ref{ultracyril} (as well as the proof of Fact \ref{p3.3}) uses the following result of Popa \cite[Lemma 2.3]{popa1981problem}:

\begin{fact}\label{lem: Popa81}
Let $M$ be a $\sigma$-finite von Neumann algebra, let $\varphi$ be a normal faithful state on $M$, and let $N$ be a von Neumann subalgebra of $M_{\varphi}$. Fix $\varepsilon>0$ and suppose that $x\in M\setminus \{0\}$ is such that $E_N^{\varphi}(x)=0$. Then there exists $u\in \mathbf{U}(N)$ such that 
\[\|uxu^*-x\|_{\varphi}^2>(2-\varepsilon)\|x\|_{\varphi}^2.\] 
\end{fact}

The argument below was  communicated to the first author by Cyril Houdayer during the joint work that resulted in \cite{AHHM20}. This was later extended to \cite[Proposition 3.3]{AHHM20}  mentioned above. We thank him for his permission to include the proof here.

\begin{proof}[Proof of Proposition \ref{ultracyril}]
For each $i\in I$, let $\E_i:M_i\to M_{\psi_i}$ denote the canonical conditional expectation and set $\E:=\prod_\cU \E_i:M\to \prod_\cU M_{\psi_i}$.  We first claim that if $x=(x_i)_\cU\in M_\psi'\cap M$ satisfies $\E(x)=0$, then $x=0$.  Without loss of generality, we may suppose that $\E_i(x_i)=0$ for all $i\in I$. By Fact \ref{lem: Popa81} (applied to $N=M_{\psi_i}, \varepsilon=1$), for each $i\in I$, there exists $u_i\in \mathbf{U}(M_{\psi_i})$ such that 
\[\|u_ix_iu_i^*-x_i\|_{\psi_i}^2\geq \|x_i\|_{\psi_i}^2.\]
Since each $M_{\psi_i}$ is finite, the sequence $(u_i)_{i\in I}$ defines an element $u=(u_i)_{\cU}\in \mathbf{U}(\prod_\cU M_{\psi_i})\subseteq \mathbf{U}(M_\psi)$.  We then have 
$\|x\|_{\psi}^2\leq \|uxu^*-x\|_{\psi}^2=0,$
whence $x=0$. 

We next claim that $(\prod_\cU M_{\psi_i})'\cap M=\mathbb C$.  To see this, suppose that $x\in (\prod_\cU M_{\psi_i})'\cap M$. Then $\E(x)\in (\prod_\cU M_{\psi_i})'\cap M$.  Indeed, if $a\in \prod_\cU M_{\psi_i}$, then $xa=ax$, whence
$$\E(x)a=\E(xa)=\E(ax)=a\E(x).$$
It follows that $\E(x)\in \mathcal{Z}(\prod_{\cU}M_{\psi_i})$.  Since each $M_{\psi_i}$ is a finite factor, we have $\mathcal{Z}(\prod_{\cU}M_{\psi_i})=\prod_{\cU}\mathcal{Z}(M_{\psi_i})$. Therefore 
\[\E(x)\in \prod_{\cU}\mathcal{Z}(M_{\psi_i})\subseteq \prod_{\cU}((M_i)_{\psi_i}'\cap M_i)=\mathbb{C}.\]
By the first claim, it follows that 
$$x=(x-\E(x))+\E(x)=0+\E(x)\in \mathbb{C},$$
as desired.

Since $\prod_{\cU}M_{\psi_i}\subseteq M_{\psi}$, we can conclude that 
$M_{\psi}'\cap M=\mathbb{C}$. 
\end{proof}

\begin{cor}\label{ultracyrilcorollary}
    Let \((M_i,\varphi_i)_{i\in I}\) be a family of W\(^*\)-probability spaces with each \(M_i\) a type III\(_1\) factor with separable predual and trivial bicentralizer. Let \(\cU\) be a countably incomplete ultrafilter on \(I\) and let \((M,\varphi)=\prod_{\cU}(M_i,\varphi_i)\). Then \(M_{\varphi}'\cap M=\C.\)  
\end{cor}

\begin{proof}
By Fact \ref{bicentralizerlargecentralizer}, for each $i\in I$, there is a faithful normal state $\psi_i$ on $M_i$ such that $M_{\psi_i}'\cap M_i=\bb C$.  Set $(M,\psi):=\prod_\u (M_i,\psi_i)$.  By Proposition \ref{ultracyril}, we have that $M_\psi'\cap M=\bb C$.  However, since $\u$ is countably incomplete, $M$ has strictly homogeneous state space, whence $M_\varphi'\cap M=\bb C$.
\end{proof}

We can use Proposition \ref{ultracyril} to give a different proof of the axiomatizability of the class of III$_1$ factors with trivial bicentralizer.  As in the proof of Proposition \ref{pseudoperiodicequivalence}, let $T$ denote the sentences true in all III$_1$ factors with trivial bicentralizer.  It suffices to show that a model of $T$ has trivial bicentralizer.  Suppose $(N,\psi)$ is a model of $T$; by Lemma \ref{separableenoughbicentralizer}, we may assume that $(N,\psi)$ is separable.  As in the proof of Proposition \ref{pseudoperiodicequivalence}, there is a countable family $(M_n,\varphi_n)_{n\in \bb N}$ of III$_1$ factors with trivial bicentralizer and an ultrafilter $\cU$ on $\mathbb N$ such that $(N,\psi)\equiv (M,\varphi):=\prod_\cU (M_n,\varphi_n)$.  (One can indeed take a countable such family by replacing $T$ with a countable dense subset of $T$.)  Without loss of generality, by Corollary \ref{bicentralizeraxiomatizable} and Downward L\"owenheim-Skolem, we may assume that each $M_n$ is separable.  By Fact \ref{bicentralizerlargecentralizer}
and Connes-Stormer transitivity (see also \cite[Proposition 3.12]{goldbringhoudayer}), we may assume that each $\varphi_n$ has large centralizer.  By Proposition \ref{ultracyril}, $\varphi$ has large centralizer, whence $(M,\varphi)$ has trivial bicentralizer.  Since $(N,\psi)$ admits an elementary embedding into $(M,\varphi)$ (by $\aleph_1$-saturation of $(M,\varphi)$  \cite[Proposition 7.6]{yaacov2008model}), Corollary \ref{BCpreceq} implies that $(N,\psi)$ has trivial bicentralizer.

Fact \ref{bicentralizerlargecentralizer} states that a III$_1$ factor $M$ with separable predual and trivial bicentralizer always admits a faithful normal state $\varphi$ with large centralizer, that is, with $M_\varphi'\cap M=\bb C$.  Whether or not this result holds for nonseparable $M$ seems to be unknown.  Corollary \ref{ultracyrilcorollary} shows that the conclusion of Fact \ref{bicentralizerlargecentralizer} holds when $(M,\varphi)$ is a countably incomplete ultraproduct of III$_1$ factors with separable predual.  We can use Proposition \ref{ultracyril} to prove something even more general:

\begin{prop}
Suppose that $(M,\varphi)$ is an $\aleph_0$-saturated W$^*$-probability space with trivial bicentralizer for which $M$ is a III$_1$ factor.  Then $\varphi$ has large centralizer.
\end{prop}

\begin{proof}
Let $(N,\psi)$ be a separable elementary substructure of $(M,\varphi)$.  By Corollary \ref{BCpreceq}, $(N,\psi)$ still has trivial bicentralizer, whence, by Corollary \ref{ultracyrilcorollary}, we have that $\psi^\cU$ has large centralizer, where $\cU$ is a sufficiently good ultrafilter so that $(M,\varphi)$ admits an elementary embedding into $(N,\psi)^\cU$.

Suppose now that $a\in M_\varphi'\cap M$.  It suffices to show that $a\in (N^{\u}_{\psi^\cU})'\cap N^\cU$.  Suppose, towards a contradiction, that there is $b\in N^\u_{\psi^\cU}$ such that $\|[a,b]\|_\varphi=:\epsilon>0$.  Without loss of generality, we may assume that $b\in S_1(N^{\cU})$.  Then 
$$(N^\cU,\psi^\cU)\models \inf_{y\in S_1}\max\left(\sup_{z\in S_1}|\psi^\u(yz-zy)|,\epsilon \dminus \|[a,y]\|_{\psi^{\cU}}\right).$$
By elementarity, we have that 
$$(M,\varphi)\models \inf_{y\in S_1}\max\left(\sup_{z\in S_1}|\varphi(yz-zy)|,\epsilon \dminus \|[a,y]\|_{\varphi}\right).$$
By $\aleph_0$-saturation, such an element $y\in S_1(M,\varphi)$ must actually exist, contradicting the fact that $a\in M_\varphi'\cap M$. 
\end{proof}

\begin{cor}\label{cor generalizing Haagerup for ultraproducts}
Suppose that $(M_i,\varphi_i)_{i\in I}$ is a family of III$_1$ factors with trivial bicentralizer.  Further suppose that $\cU$ is a countably incomplete ultrafilter on $I$.  Set $(M,\varphi):=\prod_\cU (M_i,\varphi_i)$.  Then $\varphi$ has large centralizer.
\end{cor}

\begin{remark}
One cannot remove the assumption that the ultrafilter is countably incomplete in Corollary \ref{cor generalizing Haagerup for ultraproducts}.  Indeed, suppose that $\u$ is a countably complete ultrafilter and $(M,\varphi)$ is a W$^*$-probability space with $M$ a III$_1$ factor with separable predual and trivial bicentralizer and $\varphi$ any faithful normal state on $M$ that does not have large centralizer (such as an ergodic state).  Then the ultrapower $(M,\varphi)^\u$ is isomorphic to $(M,\varphi)$ and thus $\varphi^\u$ does not have large centralizer.
One might instead ask for a variant of Corollary \ref{cor generalizing Haagerup for ultraproducts} for arbitrary ultrafilters that merely asks if the ultraproduct has a state with large centralizer.  However, in some sense, this is tantamount to asking if all III$_1$ factors with trivial bicentralizer have a state with large centralizer.  Indeed, if there is a countably complete ultrafilter $\u$, then there is a measurable cardinal $\kappa$ (see \cite[Proposition 17.3.4]{goldbring2022ultrafilters}).  If $M$ is a III$_1$ factor with trivial bicentralizer and with predual of density character less than $\kappa$ and $\varphi$ is any state on $M$, the argument in the preceding section shows that the ultrapower $(M,\varphi)^\u$ is isomorphic to $M$, so asking that this ultrapower have a state with large centralizer is the same as asking that $M$ itself has a state with large centralizer.
\end{remark}

\subsection{Axioms}

We now work towards giving concrete axioms for the above classes.  Fix an axiomatizable class $\cal K$ of diffuse W$^*$-probability spaces.

We begin with a calculation:

\begin{lem}\label{lem v*av-u*au}
Suppose that $(M,\varphi)$ is a W$^*$-probability space and fix $\eta>0$. For all $u\in \mathbf{U}(M)$, $a\in S_1(M)$ and contraction $v\in \tilde{\mathbf{U}}_{1+\eta}(M)$, the following inequality holds:
\begin{equation*}
\|v^*av-u^*au\|_{\varphi}\le (2+\eta)\sqrt{2}\|v-u\|_{\varphi}^{\#}.
\end{equation*}
\end{lem}
\begin{proof}
    Since $a\in S_1(M)$ and $v\in \tilde{\mathbf{U}}_{1+\eta}(M)$, we have 
    \[\|(v^*-u^*)av\|_{\varphi}\le (1+\eta)\|(v^*-u^*)a\|_{\varphi}\le (1+\eta)\|v^*-u^*\|_{\varphi}.\]
    Therefore, by using the inequality $\|x\|_{\varphi}\le \sqrt{2}\|x\|_{\varphi}^{\#}$, we obtain
    \begin{align*}
        \|v^*av-u^*au\|_{\varphi}&\le \|(v^*-u^*)av\|_{\varphi}+\|u^*a(v-u)\|_{\varphi}\\
        &\le (1+\eta)\|v^*-u^*\|_{\varphi}+\|u^*a\|\|v-u\|_{\varphi}\\
        &\le \sqrt{2}(2+\eta)\|v-u\|_{\varphi}^{\#}.
    \end{align*}
\end{proof}

\begin{prop}\label{boundonconvex}
For all $\ep,\eta>0$, there is $N=N(\ep,\eta)$ such that:  whenever $(M,\varphi)\in \cal K$ has trivial bicentralizer and $a\in S_1(M)$, there are $u_1,\ldots,u_N\in \tilde{\mathbf{U}}_{1+\eta}(M)$ with \begin{equation}\|\frac{1}{N}\sum_{i=1}^N u_i^*au_i-\varphi(a)\cdot 1\|_\varphi<\ep.\label{eq N combinations of u*au}
\end{equation}
\end{prop}

\begin{proof}
First, we observe that it suffices to show the following slightly weaker statement: for all $\ep,\eta>0$, there exists $N_0=N_0(\ep,\eta)\in \N$ such that whenever $(M,\varphi)\in \mathcal{K}$ has trivial bicentralizer and $a\in S_1(M)$, there are $N\le N_0$ and $u_1,\dots, u_N\in \tilde{\mathbf{U}}_{1+\eta}(M)$ such that the inequality (\ref{eq N combinations of u*au}) holds. Indeed, for each $\ep,\eta>0$, we may then set $N(\ep,\eta)=N_0(\ep,\eta)!$, and if $(M,\varphi)\in \mathcal{K}$ satisfies $\BC(M,\varphi)=\C$ and $a\in S_1(M)$, then we may find $N_1\le N_0(\ep,\eta)$ and $u_1,\dots,u_{N_1}\in \tilde{\mathbf{U}}_{1+\eta}(M)$ such that (\ref{eq N combinations of u*au}) holds where $N$ is replaced by $N_1$. Set $m:=N(\ep,\eta)/N_1\in \N$.  Then the $N(\ep,\eta)$ elements $u_{i,j}:=u_i\, i=1,\dots,N_1,\,j=1,\dots,m$ in $\tilde{\mathbf{U}}_{1+\eta}(M)$ satisfy the  
inequality (\ref{eq N combinations of u*au}), where the $u_i$'s are replaced by the $u_{i,j}$'s. 

Suppose that the weaker statement does not hold for some $\ep,\eta>0$.  For each $n\geq 1$, take a counterexample $(M_n,\varphi_n)$ and $a_n\in S_1(M_n)$.  Choose a nonprincipal ultrafilter $\cU$ on $\N$ and set $(M,\varphi):=\prod_\u (M_n,\varphi_n)$ and $a:=(a_n)_\u\in S_1(M)$.  By Corollary \ref{bicentralizeraxiomatizable}, $(M,\varphi)$ has trivial bicentralizer, whence we may find unitaries $u_1,\ldots,u_N\in M$ such that $$\|u_i\varphi-\varphi u_i\|<\delta\left(\frac{\ep}{2\sqrt{2}(2+\eta)},\eta\right)$$ and $$\|\frac{1}{N}\sum_{i=1}^N u_i^*au_i-\varphi(a)\cdot 1\|_\varphi<\ep/2,$$ where $\delta(\frac{\ep}{2\sqrt{2}(2+\eta)},\eta)$ is as in Proposition \ref{almostcentralizingalmostunitary}.  Write $u_i=(u_{i,n})_\u$ with each $u_{i,n}$ a unitary in $M_n$ and note that $\|u_{i,n}\varphi_n-\varphi_n u_{i,n}\|<\delta(\frac{\ep}{2\sqrt{2}(2+\eta)},\eta)$ for $\u$-almost all $n$.  By the definition of $\delta$, for these $n$ there are contractions $v_{i,n}\in \tilde{\mathbf{U}}_{1+\eta}(M_n)$ such that $\|u_{i,n}-v_{i,n}\|_{\varphi_n}^\#<\frac{\ep}{2\sqrt{2}(2+\eta)}$.  By Lemma \ref{lem v*av-u*au}, we have $$\|v_{i,n}^*a_nv_{i,n}-u_{i,n}^*a_nu_{i,n}\|_{\varphi_n}<\frac{\ep}{2}$$ for $\cU$-almost all $n$ and $i=1,\dots, N$. Therefore, for $\u$-almost all $n$, we have $\|\frac{1}{N}\sum_{i=1}^N v_{i,n}^*a_nv_{i,n}-\varphi_n(a_n)\cdot 1\|_{\varphi_n}<\ep$. In particular, we may find such $n\ge N$, which is a contradiction to the choice of $a_n$.
\end{proof}

\begin{lem}
Let $\Phi_\eta(y)$ be the formula which is the maximum of the following two formulae:
$$\max(\|y^*y-1\|_\varphi,\|yy^*-1\|_\varphi)\dotminus \eta$$ $$\sup_{a\in S_1}\max(\varphi(y^*a^*ay)\dotminus (1+\eta)^2\varphi(a^*a),\varphi(ya^*ay^*)\dotminus (1+\eta)^2\varphi(a^*a)).$$
Then $\tilde{\mathbf{U}}_{1+\eta}$ is the zeroset of $\Phi_{\eta}(y)$.
\end{lem}

\begin{proof}
This follows immediately from Lemma \ref{lem Kbounded}.
\end{proof}

For each $\epsilon>0$ and $0<\eta<1$, set $N:=N(\epsilon,\eta)$ (as in Proposition \ref{boundonconvex}) and let $\theta_{\epsilon,\eta}$ denote the sentence
$$\sup_{x\in S_1} \inf_{u_1,\ldots,u_N\in S_2}\max(\max_i \Phi_\eta(u_i),\|\frac{1}{N}\sum_{i=1}^N u_i^*xu_i-\varphi(x)\cdot 1\|_\varphi\dotminus \epsilon).$$

\begin{thm}\label{thm trivialBCaxiomatizable}
Suppose $T_{\cal K}$ is a set of axioms for $\cal K$.  Then 
$$T_{\cal K}\cup\{\theta_{\epsilon,\eta}=0 \ : \ \epsilon,\eta>0\}$$
axiomatizes the class of elements of $\cal K$ with trivial bicentralizer.
\end{thm}

\begin{proof}
By the previous proposition, the elements of $\cal K$ with trivial bicentralizer model these axioms.  Conversely, suppose that $(M,\varphi)$ models these axioms.  By Corollary \ref{BCpreceq}, it is enough to show that $(M,\varphi)^\u$ has trivial bicentralizer, where $\u$ is any nonprincipal ultrafilter on $\bb N$.  Fix $a\in S_1(M,\varphi)^\u$ and $\eta>0$.  Let $u_1,\ldots,u_N\in S_2(M,\varphi)^\u$ realize the infimum; this is possible by countable saturation of the ultraproduct.  Then $u_i\in \tilde{\mathbf{U}}_{1+\eta}((M,\varphi)^\u)$ for all $i$.  By Proposition \ref{prop tbprop1.3}, we have that $(M,\varphi)^\u$ has trivial bicentralizer, as desired.
\end{proof}

\begin{remark}
Suppose that $\cal K$ is $\forall\exists$-axiomatizable, such as in the case that $\cal K$ is the class of III$_1$ factors.  Note then that the axiomatization we just gave is $\forall\exists\forall$, even though, by Corollary \ref{bicentralizeraxiomatizable}, there is in principle a set of $\forall\exists$-axioms for this class.
\end{remark}

\section{Is the bicentralizer a zeroset?}\label{sec BCP and zeroset}

Recall from Subsection \ref{BCsection} that if there is a counterexample to the bicentralizer problem, then there is a counterexample with a self-bicentralizing state.  This allows us to prove the following equivalent formulation of the bicentralizer problem:

\begin{thm}\label{prop zero iff trivialBC}
    The following are equivalent: 
    \begin{enumerate}
        \item The bicentralizer problem has a positive solution. 
        \item For every family $(M_i,\varphi_i)_{i\in I}$ of W$^*$-probability spaces with $M_i$ a $\mathrm{III}_1$ factor and every ultrafilter $\u$ on $I$, one has 
        $$\prod_\u \BC(M_i,\varphi_i)\subseteq \BC\left(\prod_\u (M_i,\varphi_i)\right).$$
        \item The same as (2), but only for ultra\emph{powers} with respect to countable index sets.
    \end{enumerate}
\end{thm}
\begin{proof}
    The implication (1) implies (2) follows from Remark \ref{rem separableunnecessary} while the implication (2) implies (3) is obvious. Assume, towards a contradiction, that (3) holds and yet the bicentralizer problem has a negative solution.  By the previous discussion, there is a III$_1$ factor $M$ which has a self-bicentralizing state $\varphi$. Fix a nonprincipal ultrafilter $\u$ on $\N$. By (3), we have 
    \[(M^{\u},\varphi^{\u})=(\BC(M,\varphi)^{\u},\varphi^{\u})\subseteq \BC(M^{\u},\varphi^{\u})\subseteq (M^{\u},\varphi^{\u}),\]
    whence 
    \[(M^{\u},\varphi^{\u})=\BC(M^{\u},\varphi^{\u}).\]
    This implies that $\varphi^{\u}$ is a self-bicentralizing state on $M^{\u}$, hence it is an ergodic state. However, this contradicts the fact that $M^{\u}$ is a type III$_1$ factor with strictly homogeneous state space, so that  $M^{\u}_{\varphi^{\u}}$ is a II$_1$ factor (see \cite[Proposition 4.24]{AH14}). 
\end{proof}

Condition (2) in the previous theorem has model-theoretic meaning, namely, it means that the bicentralizer is a \textbf{$T_{\mathrm{III}_1}$-zeroset}, where $T_{\mathrm{III}_1}$ is the theory of III$_1$ factors.  For an explanation of the relevant model theory, see Appendix \ref{definabilitysection}.  As such, we can summarize Theorem \ref{prop zero iff trivialBC} as saying:  the bicentralizer problem has a positive solution if and only if the bicentralizer is a $T_{\mathrm{III}_1}$-zeroset.

We can use the previous theorem to give a quantitative reformulation of the bicentralizer problem having a positive solution.  Before doing so, we need the following:

\begin{lem}\label{saturatedlemma}
Suppose that $(M,\varphi)$ is an $\aleph_1$-saturated W$^*$-probability space.  Then whenever $a\notin \BC(M,\varphi)$, there is $u\in \mathbf{U}(M_\varphi)$ with $u^*au\not=a$.
\end{lem}

\begin{proof}
Since $a\notin \BC(M,\varphi)$, there is $\epsilon>0$ such that $\inf_{\delta>0}\epsilon_M(a,\delta)\geq \epsilon$, meaning that, for each $n\geq 1$, there is a unitary $u_n\in M$ such that $\|u_n\varphi-\varphi u_n\|\leq 1/n$ and $\|u_n^*au_n-a\|_\varphi \geq \epsilon$.  By Lemma \ref{lem asymptotic centralizer}, we may consider $u:=(u_n)_\u\in M^\cU_{\varphi^\u}$.  Then $u$ belongs to the centralizer of $\varphi^\u$ and is in particular totally $1$-bounded.

Now consider the partial type $\Gamma(x)$, where $x$ ranges over $S_1$, consisting of the following formulae:
\begin{itemize}
    \item $\max(d(xx^*,1),d(x^*x,1))=0$
    \item $\epsilon\dotminus \|x^*ax-a\|_\varphi=0$
    \item $\sup_{b\in S_1}|\varphi(bx)-\varphi(xb)|\dotminus 1/n=0$, one such condition for each $n\geq 1$.
\end{itemize}
Then $\Gamma$ is realized in $M^\u$ by $u$, whence $\Gamma$ is finitely satisfiable in $M$.  Since $M$ is $\aleph_1$-saturated, it follows that $\Gamma$ is realized in $M$, as desired.
\end{proof}

\begin{thm}\label{bicentzeroset}
The following are equivalent:
\begin{enumerate}
    \item The bicentralizer problem has a positive solution.
    \item For all $\epsilon>0$, there is $\delta=\delta(\epsilon)$ such that, for all $(M,\varphi)\models T_{\mathrm{III}_1}$ and $a\in S_1(\BC(M,\varphi))$, we have $\tilde{\epsilon}_M(a,\delta)<\epsilon$.
    \item The same as (2), but with $\tilde{\epsilon}_M$ replaced by $\epsilon_M$.
\end{enumerate}
\end{thm}

\begin{proof}
Since $\|u^*u-1\|_\varphi\leq \delta$ whenever $u\in \tilde{\mathbf{U}}_{1+\delta}$, the implication (1) $\Rightarrow$ (2) follows immediately by taking $\delta:=\epsilon/2$.  The implication (1) $\Rightarrow$ (3) trivially holds when $\delta:=\epsilon$.

Now suppose that (2) holds.  We verify condition (3) in Theorem \ref{prop zero iff trivialBC}.  Towards this end, take a family $(M_n,\varphi_n)_{n\in \bb N}$ of W$^*$-probability spaces, where each $M_n$ is a $\mathrm{III}_1$ factor, an ultrafilter $\u$ on $\bb N$, and elements $a_n\in S_1(\BC(M_n,\varphi_n))$; it suffices to show that $a:=(a_n)_\u\in \BC(M,\varphi)$, where $(M,\varphi):=\prod_\u (M_n,\varphi_n)$.  If the desired conclusion fails, then by Lemma \ref{saturatedlemma}, there is $u\in \mathbf{U}(M_\varphi)$ and $\epsilon>0$ such that $\|u^*au-a\|_\varphi>\epsilon$.  Set $\delta:=\delta(\epsilon)$ and write $u=(x_n)_\u$ with $x_n\in \tilde{\mathbf{U}}_{1+\delta}(M_n)$, which is possible by Lemma \ref{lem representing u by 1+deltabdd}.  Then $\|x_n^*a_nx_n-a_n\|_{\varphi_n}>\epsilon$ for $\u$-almost all $n$, which shows $\tilde{\epsilon}_{M_n}(a_n,\delta)> \epsilon$ for $\u$-almost all $n$, contradicting the definition of $\delta$.

The proof that (3) implies (1) is similar.  Using the notation of the previous paragraph, one can write $u=(u_n)_\u$ with each $u_n\in \mathbf{U}(M_n)$.  Since $$\|u\varphi-\varphi u\|=\lim_{n\to \u}\|u_n\varphi_n-\varphi_n u_n\|,$$ we have that $\|u_n\varphi_n-\varphi_n u_n\|<\delta$ for $\u$-almost all $n$, which shows $\epsilon_{M_n}(a_n,\delta)> \epsilon$ for $\u$-almost all $n$, contradicting the definition of $\delta$.
\end{proof}

\appendix 
\section{Proof of Proposition \ref{prop compactspec}}\label{appendixA}
In this appendix, we give a proof of Proposition \ref{prop compactspec}. 
The first lemma is well-known in Schwartz distribution theory. 
\begin{lem}\label{lem Linfty is tempered}
Every $f\in L^\infty(\mathbb R)$ defines a tempered distribution. More precisely,
the map
\[
L^\infty(\mathbb R)\ni f \longmapsto T_f\in \mathcal S'(\mathbb R),
\qquad
\langle T_f,\phi\rangle:=\int_{\mathbb R} f(t)\phi(t)\,dt
\]
is a well-defined linear embedding.
\end{lem}

\begin{proof}
Fix $f\in L^\infty(\mathbb R)$. For each $\phi\in \mathcal S(\mathbb R)$, since
$\phi$ is rapidly decreasing, it belongs to $L^1(\mathbb R)$, whence the
integral
\[
\langle T_f,\phi\rangle:=\int_{\mathbb R} f(t)\phi(t)\,dt
\]
is well-defined and satisfies
\[
|\langle T_f,\phi\rangle|
\le \|f\|_{L^\infty}\|\phi\|_{L^1}.
\]
Thus $T_f$ is a linear functional on $\mathcal S(\mathbb R)$.

It remains to prove continuity for the Schwartz topology. For every
$\phi\in \mathcal S(\mathbb R)$,
\[
\|\phi\|_{L^1}
=
\int_{\mathbb R} |\phi(t)|\,dt
=
\int_{\mathbb R} \frac{1}{1+t^2}(1+t^2)|\phi(t)|\,dt
\le
\left(\int_{\mathbb R}\frac{dt}{1+t^2}\right)
\sup_{t\in\mathbb R}(1+t^2)|\phi(t)|.
\]
Hence
\[
|\langle T_f,\phi\rangle|
\le
\|f\|_{L^\infty}\|\phi\|_{L^1}
\le
\pi \|f\|_{L^\infty}\sup_{t\in\mathbb R}(1+t^2)|\phi(t)|.
\]
The map
\[
\phi\longmapsto \sup_{t\in\mathbb R}(1+t^2)|\phi(t)|
\]
is one of the standard seminorms defining the Schwartz topology on
$\mathcal S(\mathbb R)$. Therefore $T_f$ is continuous on $\mathcal S(\mathbb R)$,
that is, $T_f\in \mathcal S'(\mathbb R).$

Finally, if $T_f=0$ in $\mathcal S'(\mathbb R)$, then
\[
\int_{\mathbb R} f(t)\phi(t)\,dt=0
\qquad
\]
for all $\phi\in \mathcal S(\mathbb R)$.
Since $C_c^\infty(\mathbb R)\subset \mathcal S(\mathbb R)$, it follows that
$f=0$ as a distribution, hence $f=0$ almost everywhere. Therefore the map
$f\mapsto T_f$ is injective.
\end{proof}

The next lemma is certainly known, but we were unable to find a reference in the literature and so we provide a proof.
\begin{lem}\label{lem Arvesonsp}
Let $M$ be a von Neumann algebra, let $\alpha=(\alpha_t)_{t\in\mathbb R}$ be a
$\sigma$-weakly continuous one-parameter automorphism group on $M$, and fix $a>0$. Then the following two conditions are equivalent: 
\begin{itemize}
    \item[{\rm (1)}] $x\in M(\alpha,[-a,a])$. 
    \item[{\rm (2)}] $\alpha_f(x)=0$ for every $f\in L^1(\mathbb R)$ such that $\operatorname{supp}(\hat{f})\cap [-a,a]=\emptyset$. 
\end{itemize}
\end{lem}
\begin{proof}
    We first show that (1) implies (2). Take $f\in L^1(\mathbb R)$ such that $\operatorname{supp}(\hat{f})\cap [-a,a]=\emptyset$. Then since $\pm a\notin \operatorname{supp}(\hat{f})$, there exists $r>0$ such that $\operatorname{supp}(\hat{f})\cap [-(a+r),a+r]=\emptyset$. Let $g=D_{r+a,a}\in L^1(\R)$, which has the property that $\hat{g}=1$ on $[-a,a]$ and $\operatorname{supp}(\hat{g})\subset [-(a+r),a+r]$. Since $\operatorname{Sp}_{\alpha}(x)\subset [-a,a]$, we have $x=\alpha_g(x)$, and thus 
    \[\alpha_f(x)=\alpha_f(\alpha_g(x))=\alpha_{f*g}(x).\]
    Since $\widehat{f*g}=\hat{f}\hat{g}=0$, we have $f*g=0$ and thus $\alpha_f(x)=0$.\\
    
    We now prove that (2) implies (1).  Take $t\in \mathbb{R}\setminus [-a,a]$; we show that $t\notin \operatorname{Sp}_{\alpha}(x)$. Take $r>0$ such that $[t-r,t+r]\cap [-a,a]=\emptyset$. Choose $f\in L^1(\mathbb{R})$ such that $\hat{f}(t)>0$ and $\operatorname{supp}(\hat{f})\subset [t-r,t+r]$. Then $\alpha_f(x)=0$ by the hypothesis in $(2)$, but $\hat{f}(t)\neq 0$, which by definition of the Arveson spectrum asserts $t\notin \operatorname{Sp}_{\alpha}(x)$.  
\end{proof}
\begin{lem}\label{lem arvesoncptspec}
Let $M$ be a von Neumann algebra and let $\alpha=(\alpha_t)_{t\in\mathbb R}$ be a
$\sigma$-weakly continuous one-parameter automorphism group on $M$.  Fix $a>0$ and
$x\in M(\alpha,[-a,a])$.  For $\omega\in M_*$, define
\[
f_\omega(t):=\omega(\alpha_t(x)).
\]
Then $f_\omega\in C_b(\mathbb R)$,
and the Fourier transform of $f_\omega$ as a tempered distribution satisfies
\[
\operatorname{supp}(\widehat{f_\omega})\subset [-a,a].
\]
\end{lem}

\begin{proof}
Since $t\mapsto \alpha_t(x)$ is $\sigma$-weakly continuous,
the scalar function $f_\omega$
indeed belongs to $C_b(\mathbb R)$.
Hence $f_\omega$ defines a tempered distribution by
\[
\langle f_\omega,\phi\rangle:=\int_{\mathbb R} f_\omega(t)\phi(t)\,dt,
\qquad \phi\in \mathcal S(\mathbb R).
\]

We must show that $\operatorname{supp}\widehat{f_\omega}\subset [-a,a]$.
By definition of the support of a tempered distribution, it is enough to prove
that
\[
\langle \widehat{f_\omega},\psi\rangle =0
\]
for every $\psi\in \mathcal S(\mathbb R)$ such that
\[
\operatorname{supp}(\psi)\cap [-a,a]=\emptyset.
\]

So let $\psi\in \mathcal S(\mathbb R)$ satisfy
$\operatorname{supp}\psi\cap [-a,a]=\emptyset$. Then, by definition of the
Fourier transform on tempered distributions, we have
\[
\langle \widehat{f_\omega},\psi\rangle
=
\langle f_\omega,\widehat{\psi}\rangle
=
\int_{\mathbb R} f_\omega(t)\widehat{\psi}(t)\,dt.
\]
Substituting $f_\omega(t)=\omega(\alpha_t(x))$, we obtain
\[
\langle \widehat{f_\omega},\psi\rangle
=
\int_{\mathbb R}\widehat{\psi}(t)\,\omega(\alpha_t(x))\,dt.
\]

Now set
\[
g:=\widehat{\psi}\in \mathcal S(\mathbb R)\subset L^1(\mathbb R).
\]
Since $x\in M$ and $t\mapsto \omega(\alpha_t(x))$ is bounded measurable, the
formula
\[
\eta\mapsto \int_{\mathbb R} g(t)\,\eta(\alpha_t(x))\,dt
\qquad \eta\in M_*
\]
defines a bounded linear functional on $M_*$. Hence, by the duality
$M=(M_*)^*$, there exists a unique element $\alpha_g(x)\in M$ such that
\[
\eta(\alpha_g(x))
=
\int_{\mathbb R} g(t)\,\eta(\alpha_t(x))\,dt
\qquad \eta\in M_*.
\]
Applying this with $\eta=\omega$, we get
\[
\langle \widehat{f_\omega},\psi\rangle
=
\omega(\alpha_g(x)).
\]

Since
$x\in M(\alpha,[-a,a])$, by Lemma \ref{lem Arvesonsp}, one has $\alpha_h(x)=0$
for every $h\in L^1(\mathbb R)$ such that $\operatorname{supp}(\widehat{h})\cap [-a,a]=\emptyset$.  
It thus suffices to check that $\operatorname{supp}(\widehat{g})\cap [-a,a]=\emptyset$.

For Schwartz functions $\psi$, we have
\[
\widehat{\widehat{\psi}}(\lambda)=2\pi\,\psi(-\lambda).
\]
Therefore
\[
\widehat{g}(\lambda)=\widehat{\widehat{\psi}}(\lambda)=2\pi\,\psi(-\lambda),
\]
 and so $\operatorname{supp}(\widehat{g})=-\operatorname{supp}(\psi)$.
Since the interval $[-a,a]$ is symmetric and
$\operatorname{supp}(\psi)\cap [-a,a]=\emptyset$, it follows that $\operatorname{supp}(\widehat{g})\cap [-a,a]=\emptyset$, as desired.  
\end{proof}
The next result is essentially a well-known characterization of the Bernstein space $B_a^{\infty}$ (see, for example, the $p=\infty$ case of \cite[Theorem 4]{Andersen14}). Since we need the growth estimate of the analytic extension of $f$ as in the statement to control the right boundedness constant, we include the proof for completeness.
\begin{lem}\label{lem FT cpt supp}
Let $f\in C_b(\mathbb R)$ and regard $f$ as a tempered distribution. Assume that
\[
\operatorname{supp}(\widehat{f})\subset [-a,a]
\]
for some $a\ge 0$. Then $f$ extends to an entire function on $\mathbb C$, still
denoted by $f$, and for every $s\in \mathbb{R}$, we have
\[
\sup_{t\in\mathbb R}|f(t-is)|\le e^{a|s|}\|f\|_\infty.
\]
\end{lem}
The proof of Lemma \ref{lem FT cpt supp} is divided into steps. First, we need the following well-known Phragm\'en--Lindel\"of principle in half-plane version. There are several results collectively called Phragm\'en--Lindel\"of principle, and although only a slight  modification of any of the existing proofs works, we include a proof here in order not to bother the readers with modifying minor details. 
\begin{lem}\label{lem PhragmenLindelof half plane}
Let \(G\) be holomorphic in the lower half-plane
\[
        \Pi_-:=\{z\in\mathbb C\mid \operatorname{Im}z<0\}
\]
and continuous on \(\overline{\Pi_-}\). Assume that there is $M>0$ such that 
\[
        |G(t)|\le M \qquad (t\in\mathbb R)
\]
and that, for some constants \(C>0\) and \(N\in\mathbb N\),
\[
        |G(z)|\le C(1+|z|)^N,
        \qquad z\in\overline{\Pi_-}.
\]
Then
\[
        |G(z)|\le M,
        \qquad z\in\overline{\Pi_-}.
\]
The analogous statement holds in the upper half-plane.
\end{lem}

\begin{proof}
We prove the lower half-plane case. Fix \(0<\alpha<1\). For \(\varepsilon>0\), define
\[
        G_\varepsilon(z)
        :=
        G(z)\exp\bigl(-\varepsilon(1+iz)^\alpha\bigr),
        \qquad z\in\overline{\Pi_-},
\]
where \((1+iz)^\alpha\) is defined using the principal branch. This is valid
because
\[
        \operatorname{Re}(1+iz)=1-\operatorname{Im}z>0
\]
on \(\overline{\Pi_-}\).

For \(t\in\mathbb R\), we have
\[
        \operatorname{Re}(1+it)^\alpha\ge 0,
\]
and hence
\[
        |G_\varepsilon(t)|\le |G(t)|\le M.
\]
We show that on the lower semicircle \(\{z\in\overline{\Pi_-}\mid |z|=R\}\), there is a
constant \(c_\alpha>0\), independent of \(R\), such that for \(R\ge 2\), 
\[
        \operatorname{Re}(1+iz)^\alpha\ge c_\alpha R^\alpha .
\]

Let $z\in \overline{\Pi_-}$ with $|z|=R$ and $w=1+iz$. 
Then \(\operatorname{Re}(w)>0\) and we write 
\[
        w=\rho e^{i\theta},
        \qquad \rho=|w|,\qquad -\frac{\pi}{2}<\theta<\frac{\pi}{2}.
\]
Using the principal branch, we have
\[
        w^\alpha=\rho^\alpha e^{i\alpha\theta}.
\]
Therefore
\[
        \operatorname{Re}(w^\alpha)
        =
        \rho^\alpha\cos(\alpha\theta).
\]
Since \(0<\alpha<1\) and \(-\frac{\pi}{2}<\theta<\frac{\pi}{2}\),
we get \(-\frac{\alpha\pi}{2}
        <
        \alpha\theta
        <
        \frac{\alpha\pi}{2}.\)
Hence
\[
        \cos(\alpha\theta)
        \ge
        \cos\left(\frac{\alpha\pi}{2}\right)
        >
        0.
\]
It follows that
\[
        \operatorname{Re}(w^\alpha)
        \ge
        |w|^\alpha
        \cos\left(\frac{\alpha\pi}{2}\right).
\]
Since $|z|=R$, we have $|w|\ge |iz|-1=R-1\ge \frac{R}{2}$ by $R\ge 2$. 
Therefore
\[
        |w|^\alpha
        \ge
        \left(\frac{R}{2}\right)^\alpha
        =
        2^{-\alpha}R^\alpha.
\]
Combining the above estimates, we obtain
\[
       \operatorname{Re}(1+iz)^\alpha
        \geq c_\alpha R^\alpha,
        \qquad
        c_\alpha:=2^{-\alpha}\cos\left(\frac{\alpha\pi}{2}\right)>0.
\]

Therefore
\[
        |G_\varepsilon(z)|
        \le
        C(1+R)^N \exp(-\varepsilon c_\alpha R^\alpha)
        \longrightarrow 0
        \qquad (R\to\infty)
\]
uniformly on the lower semicircle.

Let
\[
        D_R^-:=\{z\in\mathbb C\mid |z|<R,\ \operatorname{Im}z<0\}.
\]
This is the lower half-disc of radius \(R\). Its boundary consists of two parts: $[-R,R]\subset \mathbb R
$ 
and
\[
        \Gamma_R^-:=\{z\in\mathbb C\mid |z|=R,\ \operatorname{Im}z\le 0\}.
\]

The function
\[
        G_\varepsilon(z)
        =
        G(z)\exp\bigl(-\varepsilon(1+iz)^\alpha\bigr)
\]
is holomorphic in \(D_R^-\) and continuous on \(\overline{D_R^-}\). Hence the
maximum modulus principle gives
\[
        \sup_{z\in D_R^-}|G_\varepsilon(z)|
        \le
        \sup_{z\in \partial D_R^-}|G_\varepsilon(z)|.
\]

On \([-R,R]\), we have already shown that \(|G_\varepsilon(t)|\le M.\) 
On \(\Gamma_R^-\), the preceding estimate gives
\[
        |G_\varepsilon(z)|
        \le
        C(1+R)^N\exp(-\varepsilon c_\alpha R^\alpha).
\]
Therefore

        \[\sup_{z\in D_R^-}|G_\varepsilon(z)|
        \le
        \max\left\{
            M,\,
            C(1+R)^N\exp(-\varepsilon c_\alpha R^\alpha)
        \right\}\xrightarrow{R\to \infty}M.
        \]
Consequently, for every fixed \(z_0\in \Pi_-\), choosing \(R>|z_0|\) and
letting \(R\to\infty\) gives
\[
        |G_\varepsilon(z_0)|
        \le
        \limsup_{R\to\infty}
        \max\left\{
            M,\,
            C(1+R)^N\exp(-\varepsilon c_\alpha R^\alpha)
        \right\}
        =
        M.
\]
Since \(z_0\in \Pi_-\) was arbitrary, we obtain
\[
        |G_\varepsilon(z)|\le M,
        \qquad z\in \Pi_-.
\]
By continuity of \(G_\varepsilon\) on \(\overline{\Pi_-}\), the same estimate
also holds on the boundary line \(\mathbb R\). Thus
\[
        |G_\varepsilon(z)|\le M,
        \qquad z\in \overline{\Pi_-}.
\]
Letting $\ep\to 0$, the lemma is proved. 
\end{proof}

\begin{lem}\label{lem one-sided spec estimate}
Let \(h\in C_b(\mathbb R)\), and regard \(h\) as a tempered distribution. Assume
that
\[
        \operatorname{supp}(\widehat h)\subset [0,b]
\]
for some \(b\ge 0\). Then \(h\) has an entire extension \(H\), and
\[
        |H(z)|\le \|h\|_\infty,
        \qquad \operatorname{Im}z\le 0.
\]
Similarly, if
\[
        \operatorname{supp}(\widehat h)\subset [-b,0],
\]
then \(h\) has an entire extension \(H\), and
\[
        |H(z)|\le \|h\|_\infty,
        \qquad \operatorname{Im}z\ge 0.
\]
\end{lem}
\begin{proof}
We prove the case
\[
        \operatorname{supp}(\widehat h)\subset [0,b].
\]
Since \(\widehat h\) is a compactly supported tempered distribution, the
Paley--Wiener theorem (see, for example, \cite[Theorem 7.3.1]{Hormander03}) gives an entire function \(H\) whose restriction
to \(\mathbb R\) agrees with \(h\) as a tempered distribution. More explicitly,
with the present Fourier transform convention,
\[
        \widehat u(\lambda)
        =
        \int_{\mathbb R}e^{it\lambda}u(t)\,dt,
\]
one may write 
\[
        H(z)
        =
        \frac{1}{2\pi}
        \left\langle \widehat h(\lambda), e^{-iz\lambda}\right\rangle .
\]
This formula is understood in the standard sense that a compactly supported
distribution may be paired with a smooth function on \(\mathbb R\).

Since \(H|_{\mathbb R}\) and \(h\) define the same tempered distribution, and
both are continuous functions on \(\mathbb R\), they agree pointwise:
\[
        H(t)=h(t),
        \qquad t\in\mathbb R.
\]
Thus
\[
        |H(t)|\le \|h\|_\infty,
        \qquad t\in\mathbb R.
\]

It remains to check the growth condition in the lower half-plane. Since
\(\widehat h\) is a distribution of finite order supported in \([0,b]\), there
exist constants \(C>0\) and \(N\in\mathbb N\) such that, for every \(\varphi\in C^{\infty}(\mathbb{R})\),
\[
        \left|\langle \widehat h,\varphi\rangle\right|
        \le
        C\sum_{j=0}^N
        \sup_{\lambda\in[0,b]}
        |\varphi^{(j)}(\lambda)|.
\]
Applying this to
\[
        \varphi_z(\lambda):=e^{-iz\lambda},
\]
we get, for \(\operatorname{Im}z\le 0\),
\[
        |\varphi_z^{(j)}(\lambda)|
        =
        |(-iz)^j e^{-iz\lambda}|
        =
        |z|^j e^{(\operatorname{Im}z)\lambda}
        \le
        |z|^j
        \qquad (0\le \lambda\le b).
\]
Therefore, there is some constant \(C'>0\) such that 
\[
        |H(z)|
        \le
        C'(1+|z|)^N,
        \qquad \operatorname{Im}z\le 0.
\]
By Lemma \ref{lem PhragmenLindelof half plane},
\[
        |H(z)|\le \|h\|_\infty,
        \qquad \operatorname{Im}z\le 0.
\]

The case
\[
        \operatorname{supp}(\widehat h)\subset [-b,0]
\]
is the same, using the upper half-plane instead. Indeed, if
\(\operatorname{Im}z\ge 0\) and \(\lambda\in[-b,0]\), then
\[
        |e^{-iz\lambda}|
        =
        e^{(\operatorname{Im}z)\lambda}
        \le 1.
\]
Thus the same argument gives
\[
        |H(z)|\le \|h\|_\infty,
        \qquad \operatorname{Im}z\ge 0.
\]
\end{proof}
We are ready to prove Lemma \ref{lem FT cpt supp}. 
\begin{proof}[Proof of Lemma \ref{lem FT cpt supp}]
First consider the lower half-plane direction. Define
\[
        h(t):=e^{-iat}f(t).
\]
Then
\[
        h\in C_b(\mathbb R),
        \qquad
        \|h\|_\infty=\|f\|_\infty.
\]
With the Fourier transform convention
\[
        \widehat u(\lambda)
        =
        \int_{\mathbb R}e^{it\lambda}u(t)\,dt,
\]
we see that in \(\mathcal S'(\mathbb R)\),
\[
        \widehat h(\lambda)
        =
        \widehat f(\lambda-a).
\]

By definition of equality of tempered distributions, this means that for every
\(\varphi\in\mathcal S(\mathbb R)\),
\[
        \langle \widehat h,\varphi\rangle
        =
        \langle \widehat f,\varphi(\cdot+a)\rangle .
\]

Indeed, recall that the Fourier transform of a tempered distribution is defined by
\[
        \langle \widehat T,\varphi\rangle
        =
        \langle T,\widehat\varphi\rangle,
        \qquad \varphi\in\mathcal S(\mathbb R).
\]
Therefore
\[
\begin{aligned}
        \langle \widehat h,\varphi\rangle
        &=
        \langle h,\widehat\varphi\rangle                                      \\
        &=
        \left\langle e^{-iat}f(t),
        \int_{\mathbb R}e^{it\lambda}\varphi(\lambda)\,d\lambda
        \right\rangle                                                         \\
        &=
        \left\langle f(t),
        \int_{\mathbb R}e^{it(\lambda-a)}\varphi(\lambda)\,d\lambda
        \right\rangle .
\end{aligned}
\]
Putting
\[
        \mu=\lambda-a,
        \qquad \lambda=\mu+a,
\]
we get
\[
        \int_{\mathbb R}e^{it(\lambda-a)}\varphi(\lambda)\,d\lambda
        =
        \int_{\mathbb R}e^{it\mu}\varphi(\mu+a)\,d\mu.
\]
Thus
\[
\begin{aligned}
        \langle \widehat h,\varphi\rangle
        &=
        \left\langle f(t),
        \int_{\mathbb R}e^{it\mu}\varphi(\mu+a)\,d\mu
        \right\rangle                                                         \\
        &=
        \langle \widehat f,\varphi(\cdot+a)\rangle .
\end{aligned}
\]
This proves
\[
        \widehat h(\lambda)=\widehat f(\lambda-a)
\]
in \(\mathcal S'(\mathbb R)\).
Therefore
\[
        \operatorname{supp}(\widehat h)
        \subset [0,2a].
\]
By Lemma \ref{lem one-sided spec estimate}, \(h\) has an entire extension \(H\) such
that
\[
        |H(z)|\le \|h\|_\infty=\|f\|_\infty,
        \qquad \operatorname{Im}z\le 0.
\]
Define
\[
        F_-(z):=e^{iaz}H(z).
\]
Then \(F_-\) is entire and, for real \(t\),
\[
        F_-(t)
        =
        e^{iat}H(t)
        =
        e^{iat}h(t)
        =
        f(t).
\]
Thus \(F_-\) is an entire extension of \(f\). For \(s\ge 0\), we have
\[
        |F_-(t-is)|
        =
        |e^{ia(t-is)}|\,|H(t-is)|
        =
        e^{as}|H(t-is)|
        \le
        e^{as}\|f\|_\infty .
\]
Hence
\[
        \sup_{t\in\mathbb R}|F_-(t-is)|
        \le
        e^{as}\|f\|_\infty,
        \qquad s\ge 0.
\]

Now consider the upper half-plane direction. Define
\[
        k(t):=e^{iat}f(t).
\]
Then
\[
        k\in C_b(\mathbb R),
        \qquad
        \|k\|_\infty=\|f\|_\infty,
\]
and
\[
        \widehat k(\lambda)
        =
        \widehat f(\lambda+a).
\]
Therefore
\[
        \operatorname{supp}(\widehat k)
        \subset [-2a,0].
\]
By Lemma \ref{lem one-sided spec estimate} again, \(k\) has an entire extension \(K\) such
that
\[
        |K(z)|\le \|k\|_\infty=\|f\|_\infty,
        \qquad \operatorname{Im}z\ge 0.
\]
Define
\[
        F_+(z):=e^{-iaz}K(z).
\]
Then \(F_+\) is entire and, for real \(t\),
\[
        F_+(t)
        =
        e^{-iat}K(t)
        =
        e^{-iat}k(t)
        =
        f(t).
\]
Thus \(F_+\) is also an entire extension of \(f\). Since \(F_-\) and \(F_+\)
agree on \(\mathbb R\), they agree on all of \(\mathbb C\) by the identity
theorem. Denote this common entire extension simply by \(f\).

If \(s<0\), put \(r=-s>0\). Then \(t-is=t+ir\), and
\[
        |f(t-is)|
        =
        |f(t+ir)|
        =
        |F_+(t+ir)|
        =
        |e^{-ia(t+ir)}|\,|K(t+ir)|
        =
        e^{ar}|K(t+ir)|
        \le
        e^{ar}\|f\|_\infty.
\]
Since \(r=|s|\), this gives
\[
        |f(t-is)|
        \le
        e^{a|s|}\|f\|_\infty
        \qquad (s<0).
\]

Combining the estimates for \(s\ge 0\) and \(s<0\), we obtain
\[
        \sup_{t\in\mathbb R}|f(t-is)|
        \le
        e^{a|s|}\|f\|_\infty
\]
for every \(s\in\mathbb R\).
\end{proof}

\begin{proof}[Proof of Proposition \ref{prop compactspec}] The implication (ii)$\implies$(i) follows from \cite[Lemma 2.5]{haagerup87} while the implication (ii)$\implies$(iii) follows from Lemma \ref{lem Kbounded}. It remains to show that (i)$\implies$(ii). 
    Assume $x\in M(\sigma^{\varphi},[-a,a])$.
Set $\alpha_t:=\sigma_t^\varphi$. For each $\omega\in M_*$, consider 
the bounded continuous scalar-valued function
\[
f_{\omega}(t):=\omega(\alpha_t(x)), \qquad t\in\mathbb R.
\]

By Lemma \ref{lem arvesoncptspec}, the Fourier transform of $f_{\omega}$ regarded as a tempered distribution is supported on $[-a,a]$. Thus, by Lemma \ref{lem FT cpt supp}, $f_{\omega}$ extends to an entire analytic function which satisfies:
\[
|f_\omega(t-is)|\le e^{a|s|}\|f_\omega\|_\infty
\le e^{a|s|}\|\omega\|\,\|x\|,
\qquad s,t\in\mathbb R.
\]
Moreover, for $\omega_1,\omega_2\in M_*$ and $\lambda \in \mathbb{C}$, the entire functions $f_{\omega_1+\omega_2}$ and $f_{\omega_1}+f_{\omega_2}$ agree on $\mathbb{R}$, whence they agree on $\mathbb{C}$. Similarly, $f_{\lambda \omega}=\lambda f_{\omega}$ on $\mathbb{C}$.  Thus, for each fixed $z\in \mathbb{C}$, the map $M_*\ni \omega\mapsto f_{\omega}(z)\in \C$ is a bounded linear functional of norm at most $e^{a|\operatorname{Im} z|}\|x\|$. Therefore, there exists a unique element $\alpha_z(x)\in M$ such that $f_{\omega}(z)=\omega(\alpha_z(x))$ for every $\omega\in M_*$. It is then clear that this map is an extension of the given flow $\alpha$. 

Since the function
\[
z\mapsto \omega(\alpha_z(x))
\]
is entire for every $\omega\in M_*$, the map $z\mapsto \alpha_z(x)$ is an $M$-valued entire function (see, for example, \cite[Appendix A.\,1]{takesakiII}) such that 
\[\|\alpha_{t-is}(x)\|\le e^{a|s|}\|x\|,\,s,t\in \mathbb{R}.\]
This proves (ii). 
\end{proof}

\section{Zerosets}\label{definabilitysection}

We summarize here the theory of definability in continuous logic as presented in the second author's article \cite{goldbringspectralgap}.  This will help explain the terminology ``zeroset'' used after Theorem \ref{prop zero iff trivialBC}.

Fix a theory $T$ in a language $L$.  (In this article, $L$ will be the language of W$^*$-probability spaces introduced in the previous subsection and $T$ a theory extending the theory of W$^*$-probability spaces, most often the theory T$_{\mathrm{III}_1}$ of III$_1$ factors).  We let $\operatorname{Mod}(T)$ denote the category whose objects are models of $T$ and whose morphisms are elementary embeddings.  We also let $\operatorname{Met}$ denote the category whose objects are bounded metric spaces and whose morphisms are isometric embeddings.

Fix a finite sequence of variables $\vec x$ (possibly ranging over different sorts).  By a \textbf{$T$-functor over $\vec x$} we mean a functor $X:\operatorname{Mod}(T)\to \operatorname{Met}$ such that, for every model $M\models T$, $X(M)$ is a closed subset of $M^{\vec x}$ and such that $X$ is given by restriction on morphisms.  The main example of a $T$-functor we will have in mind is the functor which assigns to each W$^*$-probability space its bicentralizer (restricted to the set of totally $1$-bounded elements).

A main source of $T$-functors come from $T$-formulae.  To explain this notion, one defines a pseudometric $d_{T,\vec x}$ on formulae $\varphi(\vec x)$ by $$d_{T,\vec x}(\varphi(\vec x),\psi(\vec x)):=\sup\{|\varphi(\vec a)^M-\psi(\vec a)^M| \ : \ M\models T, \vec a\in M\}.$$  Thus, two formulae $\varphi$ and $\psi$ are $d_{T,\vec x}$ close if they are uniformly close in all models of $T$; in particular, they are distance $0$ from each other if they are equivalent in all models of $T$.  One separates and completes this pseudometric space to obtain the complete metric space of \textbf{$T$-formulae}.  If $\theta(\vec x)$ is a $T$-formula, then to every model $M\models T$, one has a uniformly continuous function $\theta^M:M^{\vec x}\to \mathbb R$ whose range is contained in some bounded interval in $\bb R$ (both the modulus of uniform continuity and bound on the range of $\theta^M$ are uniform over all models of $T$); this function is called the \textbf{interpretation of $\theta$ in $M$}.  If $\theta(\vec x)$ is a $T$-formula, then one obtains a $T$-functor, the \textbf{zeroset of $\theta$}, denoted $Z(\theta)$, defined by $$Z(\theta)(M):=\{\vec a\in M^{\vec x} \ : \ \theta^M(\vec a)=0\}.$$

More generally, by a \textbf{$T$-function over $\vec x$} we mean a mapping whose domain is the set of all pairs $(M,\vec a)$ with $M\models T$ and $\vec a\in M^{\vec x}$ and whose co-domain is a bounded set in $\bb R$.  As we just saw, interpretations of $T$-formulae give examples of $T$-functions (the $T$-formula $\theta(\vec x)$ sends the pair $(M,\vec a)$ to the value $\theta^M(\vec a)$); $T$-functions of this form are called \textbf{realized}.

If $X$ is a $T$-functor, then one obtains the $T$-function $d(\vec x,X)$ which, upon input $(M,\vec a)$, returns the value $d(\vec a,X(M))$.

Given a $T$-function $\Phi(\vec x,\vec y)$ and a $T$-functor $X$, one may quantify over $X$ to obtain a new $T$-function.  More specifically, define $\sup_{\vec x\in X}\Phi(\vec x,\vec y)$ to be the $T$-function which, upon input $(M,\vec b)$, returns the value $\sup\{\Phi(M,\vec a,\vec b) \ : \ \vec a\in X(M)\}$.  One defines $\inf_{\vec x\in X}\Phi(\vec x,\vec y)$ similarly.

We say that a nonnegative $T$-function $\Phi(\vec x)$ is \textbf{almost-near} if:  for all $\epsilon>0$, there is $\delta>0$ such that, for all $M\models T$ and $\vec a\in M^{\vec x}$, if $\Phi(M,\vec a)<\delta$, then there is $\vec b\in M^{\vec x}$ with $d(\vec a,\vec b)\leq \epsilon$ for which $\Phi(M,\vec b)=0$.

Here is the main theorem about definability in continuous logic (see \cite[Theorem 2.13]{goldbringspectralgap}):

\begin{fact}\label{defsetchar}
Suppose that $X$ is a $T$-functor over $\vec x$.  The following are equivalent:
\begin{enumerate}
    \item For all $T$-formulae $\theta(\vec x,\vec y)$, the $T$-functions $\sup_{\vec x\in X}\theta(\vec x,\vec y)$ and $\inf_{\vec x\in X}\theta(\vec x,\vec y)$ are realized.
    \item The $T$-function $d(\vec x,X)$ is realized.
    \item $X$ is the zeroset of an almost-near $T$-formula.
    \item For all families $(M_i)_{i\in I}$ of models of $T$ and all ultrafilters $\u$ on $I$, we have $$X(\prod_\u M_i)=\prod_\u X(M_i).$$
\end{enumerate}
\end{fact}

Any $T$-functor satisfying the above equivalent conditions is called a \textbf{$T$-definable set}.

In order to explain the connection between Theorem \ref{prop zero iff trivialBC}, we will need a characterization of being a zeroset in a similar spirit to Fact \ref{defsetchar}.  This characterization was observed by Bradd Hart, Ward Henson, and the second author several years ago.  We thank them for their permission to include this observation here.

Once again suppose that $T$ is a theory and $X$ is a $T$-functor over $\vec x$.  We will list three properties a $T$-functor may or may  not have that will end up characterizing those $T$-functors that are zerosets.  

We recall that the space  $\mathbf{S}_{\vec x}(T)$ of complete types in $T$ over $\vec x$ is equipped with its logic topology, which is merely the induced topology from viewing $\mathbf{S}_{\vec x}(T)$ as a subset of the dual of the normed space of $T$-formulae over $\vec x$ (see \cite[ Section 2.4]{goldbringspectralgap}).  More concretely, if $(p_i)_{i\in I}$ is a family of complete types with realizations $\vec a_i$ in models $M_i$ of $T$ and $\u$ is an ultrafilter on $I$, then setting $\vec a:=(\vec a_i)_\u\in \prod_\u M_i$ and $p$  the complete type of $\vec a$ in $\prod_\u M_i$, we have that $\lim_\u p_i=p$.  

Set $\hat X:=\{p\in \mathbf{S}_{\vec x}(T) \ : \ p=\tp^M(\vec a) \text{ for some }M\models T \text{ with }\vec a\in X(M)\}$.  

The above paragraph thus yields:

\begin{lem}\label{closedlemma}
The following are equivalent:
\begin{enumerate}
    \item $\hat X$ is a closed subspace of $\mathbf{S}_{\vec x}(T)$.
    \item For any family $(M_i)_{i\in I}$ of models of $T$ and ultrafilter $\u$ on $I$, we have $$\prod_\u X(M_i)\subseteq X(\prod_\u M_i).$$
\end{enumerate}
\end{lem}

In general, one always has $X(M)\subseteq \{\vec a\in M^{\vec x} \ \tp^M(\vec a)\in \hat X\}$.  In general, one cannot recover $X$ from $\hat X$.  In the case of zerosets, such a recovery is possible as shown by the following easy lemma:

\begin{lem}
For a $T$-functor $X$, the following are equivalent:
\begin{enumerate}
\item For any model $M\models T$, we have $X(M)=\{\vec a\in M \ : \ \tp^M(\vec a)\in \hat X\}$.
\item For any pair $M,N\models T$ of models of $T$ with $M\preceq N$, we have $$X(N)\cap M=X(M).$$
\end{enumerate}
\end{lem}

\begin{proof}
It is clear that (1) implies (2).  Now assume (2) and fix $M\models T$ and $\vec a\in M^{\vec x}$ for which $\tp^M(\vec a)\in \hat X$; we wish to show that $\vec a\in X(M)$.  By definition, we have that $\tp^M(\vec a)=\tp^P(\vec b)$ for some $P\models T$ and $\vec b\in X(P)$.  There are then elementary embeddings $i:M\hookrightarrow N$ and $j:P\hookrightarrow N$ such that $i(\vec a)=j(\vec b)$.  By applying (2) twice, we have that 
$$\vec b\in X(P)\Rightarrow j(\vec b)\in X(N)\Rightarrow \vec a\in X(M),$$ as desired.
\end{proof}

In light of the second item in the previous lemma, we call a $T$-functor satisfying the assumptions in the lemma an \textbf{elementary $T$-functor}.  By \L os' theorem, zerosets are indeed elementary $T$-functors.

Finally, we note the following:

\begin{lem}
If $X=Z(\theta)$ for some $T$-formula $\theta$, then $\hat X$ is a $G_\delta$ subset of $\mathbf{S}_{\vec x}(T)$.
\end{lem}

\begin{proof}
For $n\geq 1$, let $U_n$ denote the basic open subset of $\mathbf{S}_{\vec x}(T)$ determined by the open condition $|\theta(\vec x)|<1/n$.  Then $\hat X=\bigcap_{n\geq 1} U_n$, whence $\hat X$ is $G_\delta$.  
\end{proof}

\begin{remark}\label{separableremark}
If the language is separable, then closed subsets of $\mathbf{S}_{\vec x}(T)$ are automatically $G_\delta$.
\end{remark}

We are ready for the promised characterization of zerosets:

\begin{thm}\label{zerosetcharacterization}
Suppose that $X$ is a $T$-functor over $\vec x$.  Then $X$ is a zeroset if and only if it is elementary and $\hat X$ is a closed, $G_\delta$-subset of $\mathbf{S}_{\vec x}(T)$.
\end{thm}

\begin{proof}
We have already observed that zerosets satisfy the enumerated properties.  Conversely, suppose that $X$ is an elementary $T$-functor for which $\hat X$ is a closed, $G_\delta$ subset of $S_{\vec x}(T)$.  Since $\mathbf{S}_{\vec x}(T)$ is a compact Hausdorff space, there is a continuous function $f:\mathbf{S}_{\vec x}(T)\to \bb R$ such that $\hat X=Z(f)$, where $Z(f)$ denotes the zeroset of $f$.  It is a well-known and straightforward consequence of the Stone-Weierstrauss theorem that such a continuous function $f$ must be given by evaluation at a $T$-formula $\theta$, that is, $f(p)=p(\theta)$ for all $p\in \mathbf{S}_{\vec x}(T)$.  It follows that $X(M)=Z(\theta^M)$ for all models $M$ of $T$. 
\end{proof}

We now return to bicentralizers.  A particular consequence of Corollary \ref{BCpreceq} is that, letting $T_{W^*}$ denote the theory of W$^*$-probability spaces, we have that the $T_{W^*}$-functor which assigns to $(M,\varphi)$ the bicentralizer $\BC(M,\varphi)$ (or, technically, the set of totally 1-bounded elements of $\BC(M,\varphi)$) is an elementary $T_{W^*}$-functor and is thus also an elementary $T_{\mathrm{III}_1}$-functor.  Thus, we are entitled to use the following terminology: 

\begin{defn}
Call a type $p\in \mathbf{S}_x(T_{\mathrm{III}_1})$ a \textbf{bicentralizer type} if some (equiv. any) realization of $p$ lies in the bicentralizer.  
\end{defn}

To be clear, here we are treating the bicentralizer as a $T_{\mathrm{III}_1}$-functor over $x$, with $x$ a single variable of sort $S_1$.  To reiterate the definition:  to say that $p$ is a bicentralizer type is to say that, for any W$^*$-probability space $(M,\varphi)$ with $M$ a III$_1$ factor and any $a\in S_1(M,\varphi)$ with $p=\tp^{(M,\varphi)}(a)$, we have that $a\in \BC(M,\varphi)$.

Since the language of W$^*$-probability spaces is separable, Lemma \ref{closedlemma}, Remark \ref{separableremark}, Theorem \ref{zerosetcharacterization}, together with the previous observation, yield the following:

\begin{prop}
The following are equivalent:
\begin{enumerate}
        \item The bicentralizer is a $T_{\mathrm{III}_1}$-zeroset.
    \item For all families $(M_i,\varphi_i)_{i\in I}$ of models of $T_{\mathrm{III}_1}$ and all ultrafilters $\u$ on $I$, we have $\prod_\u \BC(M_i,\varphi_i)\subseteq \BC(\prod_\u (M_i,\varphi_i))$.
    \item The bicentralizer types in $\mathbf{S}_x(T_{\mathrm{III}_1})$ form a closed set in the logic topology.
\end{enumerate}
\end{prop}

The previous proposition now explains why Theorem \ref{prop zero iff trivialBC} can be summarized as saying:  the bicentralizer problem has a positive solution if and only if the bicentralizer is a $T_{\mathrm{III}_1}$-zeroset.

\bibliographystyle{siam}
\bibliography{reference}
\end{document}